\documentclass[12pt,reqno]{amsart}
\usepackage{mathrsfs,amsfonts,amsmath}
\usepackage{bm,enumerate}
\usepackage{graphicx,subfigure}
\usepackage{color,xcolor,graphicx,subfigure}

\newtheorem{theorem}{Theorem}[section]
\newtheorem{corollary}{Corollary}[section]

\newtheorem{lemma}{Lemma}[section]
\newtheorem{proposition}{Proposition}[section]
\newtheorem{definition}{Definition}[section]
\newtheorem{remark}{Remark}[section]

\makeatletter \@addtoreset{figure}{section} \makeatother
\renewcommand{\thefigure}{\arabic{section}.\arabic{figure}}
\numberwithin{equation}{section}
\allowdisplaybreaks

\title[Two-dimensional steady supersonic exothermically reacting Euler flows]
      {
      Two-dimensional steady supersonic exothermically reacting Euler flows with strong contact discontinuity over Lipschitz \textcolor[rgb]{1.00,0.00,0.00}{wall}}
\author{Wei Xiang}
\address[W. Xiang]{Department of Mathematics, \
 \small City University of Hong Kong, Kowloon, Hong Kong, P. R. China}
\email{\tt  weixiang@cityu.edu.hk}
\author{Yongqian Zhang}
\address[Y. Zhang]{School of Mathematical Sciences, \
 \small Fudan University, Shanghai 200433, P. R. China}
\email{\tt yongqianz@fudan.edu.cn}
\author{Qin Zhao}
\address[Q. Zhao]{School of Mathematical Sciences, \
\small Fudan University, Shanghai 200433, P. R. China}
\email{\tt qinzhao11@fudan.edu.cn}


\begin{document}
\begin{abstract}
In this paper, we established the global existence of supersonic entropy solutions with a strong contact discontinuity over Lipschitz wall governed by the two-dimensional steady exothermically reacting Euler equations, when the total variation of both initial data and the slope of Lipschitz wall is sufficiently small.
Local and global estimates are developed and a modified Glimm-type functional is carefully designed. Next the validation of the quasi-one-dimensional approximation in the domain bounded by the wall and the strong contact discontinuity is rigorous justified by proving  that the difference between the average of weak solution and the solution of quasi-one-dimensional system can be bounded by the square of the total variation of both initial data and the slope of  Lipschitz wall. The methods and techniques developed here is also helpful for other related problems.
\end{abstract}
\maketitle

\section{Introduction}
We are concerned with the global existence and the quasi-one-dimensional approximation of entropy solutions of two-dimensional steady supersonic exothermically reacting Euler flows, which are governed by
\begin{equation}\label{eq-reaction-1}
\begin{cases}
(\rho u)_x+(\rho v)_y = 0,\\
(\rho u^2+p)_x+(\rho uv)_y = 0,\\
(\rho uv)_x+(\rho v^2+p)_y = 0,\\
\big((\rho E+p)u\big)_x+\big((\rho E+p)v\big)_y =0,\\
(\rho uZ)_x+(\rho vZ)_y = -\rho\phi(T)Z.
\end{cases}
\end{equation}
Here $(u,v)$ is the velocity. $p$, $\rho$, $\phi(T)$, and $Z$ stand for the scalar pressure, the density, the reaction rate, and  the fraction of unburned gas, respectively. $E$ denotes the specific total energy and is given by
\begin{equation}
E=e+\frac{1}{2}(u^2+v^2)+q_0Z,
\end{equation}
where $e$ is the specific internal energy, and  $q_0>0$ is the specific binding energy of unburned gas.

If $(\rho, S)$ are chosen as the independent variables,  then we have the following constitutive relations that
\[
(e,p,T)=(e(\rho, S),p(\rho, S),T(\rho, S)).
\]
In particular, $\partial_{\rho}p(\rho, S)>0$ and $\partial_{\rho}e(\rho, S)>0$ for $\rho>0$, and $c=\sqrt{\partial_{\rho}p(\rho, S)}$ is called the local sound speed.

For the ideal polytropic gas, the constitutive relations are
\[
p=R\rho T,\quad e=c_vT,\quad \gamma=1+\frac {R}{c_v}>1,
\]
where $R,c_v,\gamma$ are all positive constants. Then the sonic speed is given by $c=\sqrt{\gamma p/\rho}$.

In this paper,
we will study the two-dimensional steady supersonic exothermically reacting Euler flow with a strong contact discontinuity over Lipschitz wall under a BV boundary perturbation (see Fig. \ref{1}) and  assume the following:
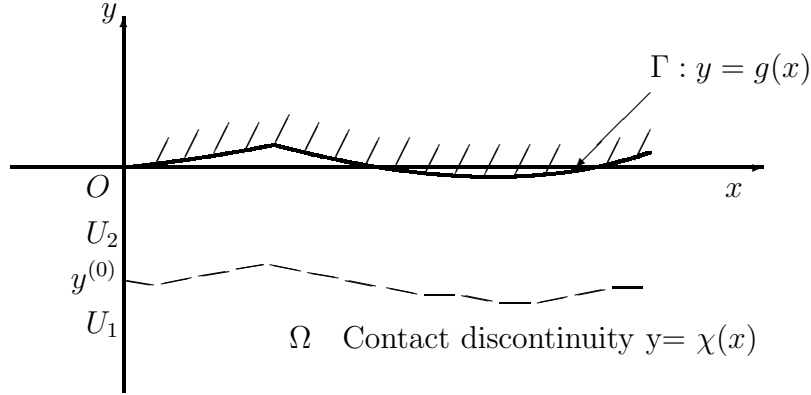
\begin{figure}[ht]
\begin{center}
\setlength{\unitlength}{1mm}
\begin{picture}(120,50)(1,4)
\linethickness{1pt}
\put(30,0){\vector(0,1){50}}
\put(15,30){\vector(1,0){100}}
\qbezier(30,30)(40,31)(50,33)
\qbezier(50,33)(80,25)(100,32)
\put(25,26){$O$}
\put(110,26){$x$}
\put(27,50){$y$}
\put(25,20){$U_2$}
\put(25,8){$U_1$}
\put(23,14){${y}^{(0)}$}
\thinlines
\put(30,15){\line(6,-1){4}}
\put(35,14.5){\line(5,1){4}}
\put(40,15.5){\line(6,1){4}}
\put(45,16.5){\line(6,1){3.6}}
\put(49.5,17){\line(5,-1){4}}
\put(54.5,16){\line(6,-1){4}}
\put(59.5,15){\line(5,-1){4}}
\put(65,14){\line(5,-1){4}}
\put(70,13){\line(1,0){4}}
\put(75,13){\line(5,-1){4}}
\put(80,12){\line(1,0){4}}
\put(85,12){\line(5,1){4}}
\put(90,13){\line(5,1){4}}
\put(95,14){\line(1,0){4}}
\put(34,30){\line(1,2){2}}
\put(38,31){\line(1,2){2}}
\put(42,32){\line(1,2){2}}
\put(46,32){\line(1,2){2}}
\put(50,33){\line(1,2){2}}
\put(54,32){\line(1,2){2}}
\put(58,31){\line(1,2){2}}
\put(62,30){\line(1,2){2}}
\put(66,30){\line(1,2){2}}
\put(70,29){\line(1,2){2}}
\put(74,29){\line(1,2){2}}
\put(78,29){\line(1,2){2}}
\put(82,29){\line(1,2){2}}
\put(86,29){\line(1,2){2}}
\put(94,30){\line(1,2){2}}
\put(98,31){\line(1,2){2}}
\put(100,40){\vector(-1,-1){10}}
\put(100,42){$\Gamma:y=g(x)$}
\put(52,6){$\Omega$\quad Contact discontinuity y= $\chi(x)$}
\end{picture}
\end{center}
\caption{Reacting Euler flow over Lipschitz wall.}\label{1}
\end{figure}

\begin{enumerate}[(H1)]
\item There exists a Lipschitz function $g(x)\in \mathrm{Lip}(\mathbb{R}_+;\mathbb{R})$ with that $g(0)=0, g'(0+)=0,$ and that $g'(x) \in BV(\mathbb{R}^+;\mathbb{R})$ such that
   \[
\Omega=\{(x,y):y<g(x), x>0\},\quad \Gamma=\{(x,y):y=g(x), x\geq 0\},
\]
and $\textbf{n}(x\pm)=\frac{(-g'(x\pm),1)}{\sqrt{(g'(x\pm))^2+1}}$ is the outer normal vectors to $\Gamma$ at the points $x\pm$, respectively.
\item The upstream flow consists of two states $U_2(y)=(u_2,v_2,p_2,\rho_2,Z_2)^\top(y)$ when ${y}^{(0)}<y<0$ and $U_1(y)=(u_1,v_1,p_1,\rho_1,Z_1)^\top(y)$ when $y<{y}^{(0)}$, which satisfy that
\[
u_i>c_i>0,\,\,0\leq Z_i\leq 1,\,\, \lim\limits_{y\to-\infty}Z_1(y)=0,
\]
where $c_i=\sqrt{\gamma{p}_i/{\rho}_i}$ is the sonic speed of state $U_i$, for $i=1,2.$
\item There exists a positive constant ${T}^{(0)}>0$, such that $T_i(y)>{T}^{(0)}$,  for $i=1,2$.
\end{enumerate}
\begin{remark}
Assumption $\mathrm{(H3)}$ is to make sure that $\phi(T)$ adimits a positive minimum value. Typically, $\phi(T)$ has the Arrhenius form which vanishes only at absolute zero temperature in \cite{ChenWagner2003}:
\[
\phi(T)=T^{\mu}e^{-\mathcal{E}/RT},
\]
where $\mu$ is a positive constant, and $\mathcal{E}$ is the action energy.
\end{remark}

The problem we are concerned with is the following initial-boundary value problem of system \eqref{eq-reaction-1} in $\Omega$ with the initial condition that
\begin{equation}\label{I}
U(0,y)=U_0(y)=
\begin{cases}
U_2(y), \qquad {y}^{(0)}<y<0,\\
U_1(y), \qquad y<{y}^{(0)},
\end{cases}
\end{equation}
and the boundary condition that
\begin{equation}\label{B}
(u,v)\cdot \textbf{n}=0\quad \text{on}\,\, \Gamma.
\end{equation}

Then we can define the global entropy solutions of problem  \eqref{eq-reaction-1} and \eqref{I}-\eqref{B}.
\begin{definition}(entropy solutions)
A $B.V.$ function $U=U(x,y)$ is called a global entropy solution of problem \eqref{eq-reaction-1} and \eqref{I}-\eqref{B} if
\begin{enumerate}
\item [$\mathrm{1)}$]
 $U$ is a weak solution of \eqref{eq-reaction-1} in $\Omega$ and satisfies \eqref{I}-\eqref{B} in the trace sense;
\item [$\mathrm{2)}$]
 $U$ satisfies the entropy inequality that
 \[
 (\rho u S)_x+(\rho v S)_y\geq \frac{q_0\rho\phi(T)Z}{T}
 \]
 in the distribution sense in $\overline{\Omega}$.
\end{enumerate}
\end{definition}

Our first result is to establish the nonlinear stability of strong contact discontinuity in the supersonic exothermically reacting Euler flows around a background solution, which is given by the case that $g(x)=0$. In this case, the problem admits a solution consisting of two constant states:
\[
U=
\begin{cases}
{U}_2^{(0)}=({u}_2^{(0)},0,{p}_2^{(0)},{\rho}_2^{(0)},0), \qquad {y}^{(0)}<y<0,\\
{U}_1^{(0)}=({u}_1^{(0)},0,{p}_1^{(0)},{\rho}_1^{(0)},0), \qquad y<{y}^{(0)},
\end{cases}
\]
where ${p}_2^{(0)}={p}_1^{(0)}$, ${u}_i^{(0)}>{c}_i^{(0)}>0,$ and the sonic speed ${c}_i^{(0)}=\sqrt{\gamma{p}_i^{(0)}/{\rho}_i^{(0)}}$, for $i=1,2$.

More precisely, we proved the following theorem.

\begin{theorem}\label{theorem-exist}
Under assumptions $\mathrm{(H1)}$-$\mathrm{(H3)}$, there exist positive constants $\delta_0$ and C, such that if
\begin{equation}\label{assum-g}
T.V.\{g'(\cdot): [0,+\infty)\}<\delta_0,
\end{equation}
and
\begin{gather}
\sup\limits_{y<{y}^{(0)}}|U_1(y)-{U}_1^{(0)}|+\sup\limits_{{y}^{(0)}<y<0}|U_2(y)-{U}_2^{(0)}|<\delta_0,\label{assum-u-1}\\
T.V.\{U_1(\cdot):(-\infty,{y}^{(0)})\}+T.V.\{U_2(\cdot):({y}^{(0)},0)\}<\delta_0,\label{assum-u-2}
\end{gather}
then the initial-boundary value problem \eqref{eq-reaction-1} and \eqref{I}-\eqref{B} admits a global entropy solution $U(x,y)\in BV_{loc}(\Omega)\cap L^{\infty}(\Omega)$ such that the following hold:
\begin{enumerate}[(i)]
\item for every $x\in [0,+\infty),$
\begin{equation}\label{estimate-tv-u}
T.V.\{U(x,\cdot): (-\infty,g(x)]\}\leq C\delta_0.
\end{equation}
\item The Lipschitz curve $\{y=\chi(x)\}$ is a strong contact discontinuity emanating from the point $(0,{y}^{(0)})$ with $\chi(x)<g(x)$ for any $x>0$, and that
\begin{equation}\label{sup-u}
\sup\limits_{y<\chi(x)}|U(x,y)-{U}_1^{(0)}|\leq C\delta_0,\quad \sup\limits_{\chi(x)<y<g(x)}|U(x,y)-{U}_2^{(0)}|\leq C\delta_0.
\end{equation}
\end{enumerate}
\end{theorem}

\medskip
Our second result is about the quasi-one-dimensional approximation.
If the flow is slowly various in the $y$ direction compared to the $x$-direction, we can introduce the quasi-one-dimensional approximation in the domain $\{ (x,y)| x>0, \, \chi(x)<y<g(x)\}$ as follows.  Neglect the changes of the solutions in $y$ direction, and let $A(x)$ be the distance between the wall and the strong contact discontinuity. Then the motion of the steady exothermically reacting Euler flows in the domain $\{ (x,y)| x>0, \, \chi(x)<y<g(x)\}$ can be described by the much simpler quasi-one-dimensional model as follows.
\begin{equation}\label{eq-QO1}
\begin{cases}
(\rho u A(x))_x = 0,\\
((\rho u^2+p)A(x))_x = A'(x)p,\\
\big((e+\frac{1}{2}u^2+\frac{p}{\rho})\rho u A(x)\big)_x =q_0A(x)\rho\phi(T)Z,\\
(\rho u Z A(x))_x = -A(x)\rho\phi(T)Z.
\end{cases}
\end{equation}

Let $(\bar{\rho}_{0},\bar{u}_{0},\bar{p}_{0},\bar{Z}_{0})^\top$ be the integral average of the initial data $U_2(y)$ in the interval ${y}^{(0)}<y<0$, that is
\[ (\bar{\rho}_{0},\bar{u}_{0},\bar{p}_{0},\bar{Z}_{0})^\top=\frac{1}{|{y}^{(0)}|}\int_{{y}^{(0)}}^{0} U_2(y)dy .\]
 Let $\bar{U}(x)=(\bar{\rho},\bar{u},\bar{p},\bar{Z})^\top$ be the integral average of the solution of system \eqref{eq-reaction-1} with respect to $y$ between the wall and the strong contact discontinuity, that is,
\[\bar{ U}(x)=\frac{1}{A(x)}\int_{\chi(x)}^{g(x)} U(x,y)dy,\]
and let $U_A(x)=(\rho_{A},u_{A},p_{A},Z_{A})^\top$ the solution of system \eqref{eq-QO1} with the initial data $U_{A,0}=(\bar{\rho}_{0},\bar{u}_{0},\bar{p}_{0},\bar{Z}_{0})^\top$. Then our second result related to the quasi-one-dimensional approximation in the domain $\{ (x,y)| x>0, \, \chi(x)<y<g(x)\}$ is as follows.

\begin{theorem}\label{theorem-Quasi-One}
Under assumptions $\mathrm{(H1)}$-$\mathrm{(H3)}$, there exist positive constants $\delta_0$ and C, such that if \eqref{assum-g}-\eqref{assum-u-2} hold,
then for any $x\geq 0$, it holds that
\begin{center}
$|\bar{U}(x)-U_A(x)|\leq C\delta_*^2$,
\end{center}
where
$
\delta_*=T.V.\{g'(\cdot): [0,+\infty)\}+T.V.\{U_1(\cdot):(-\infty,{y}^{(0)})\}+T.V.\{U_2(\cdot):({y}^{(0)},0)\}
+\sup\limits_{{y}^{(0)}<y<0}|U_2(y)-{U}_2^{(0)}|.
$
\end{theorem}

Theorem \ref{theorem-Quasi-One} justifies the validation of the quasi-one-dimensional approximation of the supersonic exothermically reacting Euler flows if $\delta_*$ is sufficiently small, \emph{i.e.}, this theorem indicates that the difference between the integral average of the weak solution of \eqref{eq-reaction-1} and the solution  of \eqref{eq-QO1} can be bounded by the square of the total variation of both initial data and the slope of Lipschitz wall.

We develop a fractional-step Glimm scheme to construct the approximate solutions to establish the global existence of the entropy solution of the initial boundary value problem \eqref{eq-reaction-1} and \eqref{I}-\eqref{B}. To make it, we have to design a Glimm-type functional based on local estimates obtained in Section \ref{sec:LE}. The key estimates are the reflection coefficient in front of the strength of the reflected 5-wave when the weak 1-wave hits the strong contact discontinuity from above governed by the corresponding homogeneous system \eqref{eq-reaction-h} is strictly less than one, as well as the exponential decay estimate of the reactant $Z$ in the reacting step.  
With the Glimm-type functional in hand, we can show that the total variation of the approximate solutions is uniformly bounded and actually small, and then by the standard argument developed in \cite{Glimm1965} to show Theorem \ref{theorem-exist}.
We remark that although elegant results had been established for the existence of entropy solutions of hyperbolic balance laws in \cite{Christoforou2015,DafermosHsiao1982, Hong2007,YingWang1980-1}, system \eqref{eq-reaction-1} concerned in this paper does not satisfy the hypotheses there. In fact, the exothermic reaction can increase the total variation of the solutions. For example, the linearized stability analysis, as well as numerical and physical experiments, have shown that certain steady detonation waves are unstable \cite{Bourlioux1991,Erpenbeck1962,Fickett1966,Lee1990}. 
However, if assume that the reaction rate function $\phi(T)$ never vanishes, then the decay estimate of the reaction plays a key role in controlling the increasing of the total variation of solutions. 

Next, in order to show Theorem \ref{theorem-Quasi-One}, we need carefully to derive several estimates on error terms of different type to pass the limit $h\rightarrow0$ such that we can get the equations that the integral average of weak solutions with respect to $y$ satisfies. Then the validation of the quasi-one-dimensional approximation is rigorously justified by applying the decay estimates of the reactant $Z$ of both system \eqref{eq-reaction-1} and \eqref{eq-QO1}, and the smallness of the $B.V.$ bounds of solutions.

The importance of the problem of steady supersonic non-reacting Euler flow past a wedge has been introduced in Courant-Friedrichs' book \cite{Courant1999}. When the flow behind the shock is smooth, the existence and asymptotic behaviour had been extensively studied by many authors (for instance, see
\cite{Chen1998-3, Chen1998-2,Chen1998-1,Gu1962,Li1980,Schaeffer1976}). Next, for the non-piecewise smooth solutions, by developing a modified Glimm scheme or  wave-front tracking scheme, 
global weak entropy solutions of the potential flow had been constructed in \cite{Zhang1999,Zhang2003,Zhang2011} when the wedge is a small perturbation of a straight wedge or a convex one. Later, global weak entropy solutions with a large shock or vortex sheet had been established for the full Euler equations in \cite{ChenZhangZhu2006,ChenZhangZhu2007}. Recently, global weak entropy solutions with transonic characteristic discontinuities had been obtained in \cite{ChenKukrejaYuan2013-1,KukrejaYuanZhao2015} when the steady supersonic non-reacting Euler flow past a convex corner surrounded by the static gas. Meanwhile, the quasi-one-dimensional approximation of isentropic or irrotational gas flow had been established in \cite{ChenGengZhang2009,GengZhang2009} by applying the Riemann semigroup via the wave-front tracking scheme (see \cite{Bressan2000,ChenXiangZhang2013} for more details of the techniques).

For the exothermically reacting Euler equations, the large-time existence of one-dimensional time-dependent entropy solutions of the Cauchy problem was established in \cite{ChenWagner2003}. 
Recently, the global existence of steady weak entropy solutions with a strong shock or strong rarefaction wave is established in \cite{ChenKuangZhang2015,ChenXiaoZhang2014}. For further information on the reacting gas dynamic theory, we refer the reader to \cite{Majda1981,Williams1985}.


The rest of this paper is organised as follows. In Section \ref{sec:LE}, several important local estimates including local interaction estimates and local estimates on the reacting step are established. In Section \ref{sec:GE}, we introduce the fractional-step Glimm scheme to construct approximate solutions and introduce a modified Glimm-type functional to prove the global estimates of the approximated solutions in the non-reacting step and the reacting step separately. Then we complete the proof of Theorem \ref{theorem-exist} in Section \ref{sec:GE}. Finally, section \ref{sec:EE} is devoted to the proof of Theorem \ref{theorem-Quasi-One}.

\section{Local estimates of solutions of the steady exothermically reacting Euler equations}\label{sec:LE}
In this section, we will establish the local wave interaction estimates for the homogeneous system, and then the local estimates on the reacting step of the steady exothermically reacting Euler equations \eqref{eq-reaction-1}.

First, system \eqref{eq-reaction-1} can be written in the following form:
\begin{equation}\label{eq-reaction-2}
W(U)_x+H(U)_y=G(U),
\end{equation}
with $U=(u ,v, p, \rho, Z)^\top$, where
\begin{align}
&W(U)=(\rho u, \rho u^2+p, \rho uv, \rho u(\frac{u^2+v^2}{2}+\frac {\gamma p}{(\gamma -1)\rho}), \rho uZ)^\top,\nonumber\\
&H(U)=(\rho v, \rho uv, \rho v^2+p, \rho v(\frac{u^2+v^2}{2}+\frac {\gamma p}{(\gamma -1)\rho}), \rho vZ)^\top,\\
&G(U)=(0, 0, 0, q_0\rho Z\phi(T), -\rho\phi(T)Z)^\top.\nonumber
\end{align}

In the case when $G(U)$ is identically zero, \eqref{eq-reaction-2} becomes the homogeneous system
\begin{equation}\label{eq-reaction-h}
W(U)_x+H(U)_y=0.
\end{equation}

\subsection{Elementary wave curves of the homogeneous system of \eqref{eq-reaction-h}}\label{secA:1}
Before deriving the local estimates, we review certain basic properties of the homogeneous system of \eqref{eq-reaction-h} and the solvability of several typical Riemann problems that appear in the process of the fractional-step Glimm scheme.

First, we remark that in this paper, $M$ is a universal constant, depending only on the data and different at each occurrence, $O(1)$ is a quantity that is bounded by $M$, and $O_{\epsilon}(U)$ is a neighbourhood with radius $M\epsilon$ and center $U$.

If $u>c$, the homogeneous system \eqref{eq-reaction-h}
 has five real eigenvalues in the $x$-direction, which are
\[
\lambda_i=\frac{uv+(-1)^{\frac{i+3}{4}} c\sqrt{u^2+v^2-c^2}}{u^2-c^2},\,\,i=1,5,\qquad
\lambda _j=\frac{v}{u},\,\, j=2,3,4.
\]
The associated linearly independent right eigenvectors are
\begin{align}
&r_i=\kappa_i(-\lambda_i,1,\rho(\lambda_i u-v), \frac{\rho (\lambda_i u-v)}{c^2},0)^\top, \quad i=1,5;\\
&r_2=(u,v,0,0,0)^\top, \quad r_3=(0,0,0,\rho,0)^\top, \quad r_4=(0,0,0,0,1)^\top,
\end{align}
where $\kappa_i$ are chosen so that $r_i\cdot\nabla \lambda _i=1$ since the $i$th-characteristic fields are genuinely nonlinear for $i=1,5$. It is easy to see that $r_j\cdot\nabla \lambda _j=0,j=2,3,4$, which means these characteristic fields are linearly degenerate. By the straightforward calculation, we have the following lemma about the value of $\kappa_i$.
\begin{lemma}\label{A.1}
At the constant state ${U}_k^{(0)}=({u}_k^{(0)},0,{p}_k^{(0)},{\rho}_k^{(0)},0)$ with ${u}_k^{(0)}>{c}_k^{(0)}>0,k=1,2$,
\[
\kappa_1({U}_k^{(0)})=\kappa_5({U}_k^{(0)})=1/(\nabla_{U}\lambda_i\cdot(-\lambda_i,1,\rho u\lambda_i, {\rho u\lambda_i}/{c^2},0)|_{U={U}_k^{(0)}})>0,\quad i=1,5.
\]
It implies that $\kappa_i(U)>0$ for any $U\in O_{\epsilon}({U}_k^{(0)})$ since $\kappa_i(U)$ are continuous for $i=1,5$.
\end{lemma}

Next, we will consider wave curves for $u>c$ in the phase space, especially in the neighborhood of ${U}_1^{(0)}$ and ${U}_2^{(0)}$. At each state $U_a=(u_a,v_a, p_a, \rho_a, Z_a)^\top$ with $u_a>c_a=\sqrt{\gamma p_a/\rho_a}$, there are five wave curves in the phase space through $U_a$.

The $j$th-contact discontinuity wave curve $C_j(U_a)$ for $j=2,3,4$, are
\[
C_j(U_a):\,\, dp=0,\qquad  vdu-udv=0.
\]
More precisely, by solving the following ODE problem
\[
\begin{cases}
\frac{dU}{d\sigma_j}=r_j(U),\, \, j=2,3,4,\\
U|_{\sigma_j=0}=U_a,
\end{cases}
\]
 we easily have that
\begin{align}
&C_2(U_a):\,\, U=(u_ae^{\sigma_2},v_ae^{\sigma_2},p_a,\rho_a,Z_a)^\top, \label{c1}\\
&C_3(U_a):\,\, U=(u_a,v_a,p_a,\rho_ae^{\sigma_3},Z_a)^\top,\label{c2}\\
&C_4(U_a):\,\, U=(u_a,v_a,p_a,\rho_a,Z_a+\sigma_4)^\top. \label{c3}
\end{align}

The $i$th-rarefaction wave curve $R_i(U_a),i=1,5$,
\begin{equation}\label{rarefaction}
R_i(U_a):\,\, dp=c^2d\rho,\,\, du=-\lambda_idv,\,\, \rho(\lambda_iu-v)dv=dp,\,\, dZ=0.
\end{equation}

The $i$th-shock wave curve $S_i(U_a),i=1,5$,
\begin{equation}\label{shock}
S_i(U_a):\, \, [p]=\frac{\hat{c}_a^2}{\hat{\gamma}}[\rho],\,\, [u]=-s_i[v], \,\, \rho_a(s_iu_a-v_a)[v]=[p],\,\, [Z]=0.
\end{equation}
where $[\cdot]$ stands for the jump of a quantity across the shock, the slope of the discontinuity
\[
s_i=\frac{u_a v_a+(-1)^{\frac{i+3}{4}}\hat{c}_a\sqrt{u_a^2+v_a^2-\hat{c}_a^2}}{u_a^2-\hat{c}_a^2},
\]
and $\hat{c}_a^2=\frac{\rho c_a^2}{\hat{\gamma}\rho_a},\hat{\gamma}=\frac{\gamma+1}{2}-\frac{(\gamma-1)\rho}{2\rho_a}$.

Following the ideas in \cite{Lax1957}, in a neighbourhood of ${U}_k^{(0)},k=1,2$, we can parameterize any physically admissible wave curves above  by
\begin{equation}
\alpha_i \mapsto \Phi_i(\alpha_i;U_a),
\end{equation}
with $\Phi_i \in C^2$, $\Phi_i|_{\alpha_i=0}=U_a$, and $\frac{\partial \Phi_i}{\partial \alpha_i}|_{\alpha_i=0}=r_i(U_a)$. For $i= 1,5$, the case $\alpha_i>0$ corresponds to a rarefaction wave, while the case $\alpha_i<0$ corresponds to a shock wave.
Moreover, $\Phi_2,\Phi_3$,$\Phi_4$ can be given with three independent parameters ($\sigma_2,\sigma_3,\alpha_4$) as
\begin{align}
&\Phi_2(\sigma_2;U_a)=(u_{a}e^{\sigma_2},v_{a}e^{\sigma_2},p_a,\rho_a,Z_a),\\
&\Phi_3(\sigma_3;U_a)=(u_a,v_a,p_a,\rho_ae^{\sigma_3},Z_a),\\
&\Phi_4(\alpha_4;U_a)=(u_a,v_a,p_a,\rho_a,Z_a+\alpha_4).\label{c3-1}
\end{align}
In particular, it holds that
\[
{U}_2^{(0)}=({u}_2^{(0)},0,{p}_2^{(0)},{\rho}_2^{(0)},0)^\top=({u}_1^{(0)} e^{\sigma_{20}},0,{p}_1^{(0)},{\rho}_1^{(0)} e^{\sigma_{30}},0)^\top.
\]

\subsection{Local interaction estimates}
Let us consider the local wave interaction estimates for the homogeneous system \eqref{eq-reaction-h} first, which include the weak wave interactions, weak wave reflections on the boundary and the interaction between the strong contact discontinuity and weak waves.

%

First, let us consider the Riemann problem only involving weak waves for \eqref{eq-reaction-h}:
\begin{equation}\label{initial-r}
U|_{x=x_0}=
\begin{cases}
U_{b}=(u_b,v_b, p_b, \rho_b, Z_b)^\top,  & y>y_0,\\
U_{a}=(u_a,v_a, p_a, \rho_a, Z_a)^\top,  & y<y_0,\\
\end{cases}
\end{equation}
where the constant states $U_a$ and $U_b$ are the below state and above state with respect to the line $y=y_0$, respectively.

Let $\tilde{\Phi}_i,i=1,2,3,5$ be the vector which only contains the first four components of $\Phi_i$, where $\Phi_i$ are defined in Section \ref{secA:1}. For the simplicity, we set
\[
\tilde{\Phi}(\alpha_5,\alpha_3,\alpha_2,\alpha_1;V_a)
=\tilde{\Phi}_5(\alpha_5;\tilde{\Phi}_3(\alpha_3;\tilde{\Phi}_2(\alpha_2;
\tilde{\Phi}_1(\alpha_1;V_a))))),
\]
with $V_a=(u_a,v_a,p_a,\rho_a)^\top,$  and $\mathcal{F}(\sigma_3,\sigma_2;V_a)=\tilde{\Phi}_3(\sigma_3;\tilde{\Phi}_2(\sigma_2;V_a))
=(u_a e^{\sigma_2},v_a e^{\sigma_2},p_a,\rho_a e^{\sigma_3})^\top$ for any  $V_a\in O_{\epsilon}({V_1}^{(0)})$ with ${V_1}^{(0)}=({u}_1^{(0)},0,{p}_1^{(0)},{\rho}_1^{(0)})^\top$.

Following the argument in \cite{Lax1957}, we easily have the following lemma.

\begin{lemma}\label{lemma-R1}
There exist positive constants $\epsilon$ and $ C$, such that for any states $U_a,U_b\in O_{\epsilon}({U}_k^{(0)}),k=1,2$, the Rieman problem \eqref{initial-r} admits a unique admissible solution of five elementary waves. In addition, the state $U_b$ can be represented by
\begin{equation}\label{probelm-R1}
\begin{cases}
&V_b=\tilde{\Phi}(\alpha_5,\alpha_3,\alpha_2,\alpha_1;V_a),\\
&Z_b=Z_a+\alpha_4,
\end{cases}
\end{equation}
with $V_b=(u_b,v_b, p_b, \rho_b)^\top$. Furthermore, it holds that $|U_b-U_a|\leq C\sum\limits_{i=1}^{5}|\alpha_i|$.
\end{lemma}

Moreover, the Glimm interaction estimates theorem (see \cite{Dafermos2010,Glimm1965,Smoller1994}) implies the following local weak wave interaction estimates.
\begin{proposition}\label{prop-in-w}
Suppose that three states $U_a, U_m$, and $U_b\in O_{\epsilon}({U}_k^{(0)}),k=1,2$, satisfy that
\begin{gather*}
V_b=\tilde{\Phi}(\gamma_5,\gamma_3,\gamma_2,\gamma_1;V_a),\quad Z_b=Z_a+\gamma_4,\\
V_b=\tilde{\Phi}(\beta_5,\beta_3,\beta_2,\beta_1;V_m),\quad Z_b=Z_m+\beta_4,\\
V_m=\tilde{\Phi}(\alpha_5,\alpha_3,\alpha_2,\alpha_1;V_a),\quad Z_m=Z_a+\alpha_4.
\end{gather*}
$($see Fig. $\mathrm{\ref{3}})$. Then it holds that
\begin{equation}\label{estimate-in-w}
\begin{cases}
&\gamma_i=\alpha_i+\beta_i+O(1)\Delta(\boldsymbol{\alpha}^{*},\boldsymbol{\beta}^{*}),\,\,i=1,2,3,5,\\
&\gamma_4=\alpha_4+\beta_4,
\end{cases}
\end{equation}
where
$\Delta(\boldsymbol{\alpha}^{*},\boldsymbol{\beta}^{*})=|\alpha_5|(|\beta_1|+|\beta_2|+|\beta_3|)
+|\beta_1|(|\alpha_2|+|\alpha_3|)
+\sum\limits_{j=1,5}\Delta_j(\alpha_{j},\beta_{j})$ with
\[
\Delta_j(\alpha_{j},\beta_{j})=
\left\{
\begin{array}{ll}
0,&\quad \mbox{$\alpha_j \ge 0$ , $\beta_j \ge 0$},\\
|\alpha_j||\beta_j|,&\quad \mbox{otherwise.}
\end{array}
\right.
\]
\end{proposition}

\begin{figure}[ht]
\begin{minipage}[t]{0.4\textwidth}
\centering
\setlength{\unitlength}{1mm}
\begin{picture}(50,72)(0,2)
\linethickness{1pt}
\put(0,0){\line(0,1){72}}
\put(24,0){\line(0,1){72}}
\put(48,0){\line(0,1){72}}
\put(0,54){\line(3,2){20}}
\put(0,54){\line(3,-1){20}}
\put(0,54){\line(5,2){20}}
\put(0,18){\line(1,1){20}}
\put(0,18){\line(5,3){20}}
\put(0,18){\line(4,-1){20}}
\put(24,40){\line(5,-2){22}}
\put(24,40){\line(3,2){22}}
\put(24,40){\line(5,1){22}}
\put(3,65){$U_b$}
\put(3,40){$U_m$}
\put(3,5){$U_a$}
\put(32,28){$U_a$}
\put(32,55){$U_b$}
\put(12,66){$\beta_5$}
\put(12,56){$\beta_{2(3,4)}$}
\put(12,46){$\beta_1$}
\put(12,36){$\alpha_5$}
\put(12,23){$\alpha_{2(3,4)}$}
\put(12,11){$\alpha_1$}
\put(40,54){$\gamma_5$}
\put(37,40){$\gamma_{2(3,4)}$}
\put(40,30){$\gamma_1$}
\end{picture}
\caption{Weak wave interactions.}\label{3}
\end{minipage}%
\begin{minipage}[t]{0.55\textwidth}
\centering
\setlength{\unitlength}{1mm}
\begin{picture}
(50,64)(0,2)
\linethickness{1pt}
\put(0,62){\line(0,-1){60}}
\put(0,18){\line(4,1){20}}
\put(0,18){\line(3,2){20}}
\put(0,62){\line(1,-1){18}}
\put(0,62){\line(5,-2){32}}
\put(24,52.3){\line(0,-1){50}}
\put(48,2){\line(0,1){56.3}}
\put(24,52.3){\line(5,-3){18}}
\put(24,52.3){\line(4,1){24}}
\put(36,55){\vector(-1,3){3}}
\put(3,60.9){\line(5,2){4}}
\put(7,59.2){\line(5,2){4}}
\put(11,57.6){\line(5,2){4}}
\put(15,56){\line(5,2){4}}
\put(19,54.4){\line(5,2){4}}
\put(11,57.6){\vector(1,2){4}}
\put(28,53.3){\line(1,2){2}}
\put(32,54.3){\line(1,2){2}}
\put(36,55.3){\line(1,2){2}}
\put(40,56.3){\line(1,2){2}}
\put(44,57.3){\line(1,2){2}}
\put(-1,65){$\Gamma_{k-1}$}
\put(49,60){$\Gamma_{k}$}
\put(11,68){$\textbf{n}_{k-1}$}
\put(33,65){$\textbf{n}_{k}$}
\put(10,53){$U_{k-1}$}
\put(38,49){$U_{k}$}
\put(23,54){$P_{k}$}
\put(2,34){$U_m$}
\put(2,10){$U_a$}
\put(36,20){$U_a$}
\put(29,51){$\omega_{k}$}
\put(10,4){$\Omega_{k-1,h}$}
\put(36,4){$\Omega_{k,h}$}
\put(36,40){$\gamma_1$}
\put(12,43){$\beta_1$}
\put(12,32){$\alpha_5$}
\put(12,18){$\alpha_{2(3,4)}$}
\end{picture}
\caption{Weak wave ~~~~~reflections on the boundary.}\label{4}
\end{minipage}
\end{figure}

Next, we consider the reflections and interactions of the waves near the boundary. Denote by $\{P_k\}_{k=0}^{\infty}$ the points $\{(x_k,y_k)\}_{k=0}^{\infty}$ in the $x$-$y$ plane with $x_k:=kh$ and $y_k:=g(kh)$. Set
\begin{align}\label{A}
&\omega_{k,k+1}=\arctan\big(\frac{y_{k+1}-y_k}{x_{k+1}-x_k}\big),
\quad \omega_k=\omega_{k,k+1}-\omega_{k-1,k},\quad \omega_{-1,0}=0,\nonumber\\
&g_{k,h}(x)=y_k+(x-x_k)\tan(\omega_{k,k+1}),\quad x\in [x_k,x_{k+1}),\\
&\Omega_{k,h}=\{(x,y):x\in [x_k,x_{k+1}), y<g_{k,h}(x)\},\quad \Omega_{h}=\bigcup_{k\ge 0}\Omega_{k,h},\nonumber\\
&\Gamma_{k}=\{(x,y):x\in [x_k,x_{k+1}), y=g_{k,h}(x)\}.\nonumber
\end{align}

Let $\textbf{n}_{k}$ be the outer normal vector to $\Gamma_{k}$, \emph{i.e.},
\begin{equation}\label{N}
\textbf{n}_{k}=\frac{(-y_{k+1}+y_k,x_{k+1}-x_k)}{\sqrt{(y_{k+1}-y_k)^2+(x_{k+1}-x_k)^2}}=(-\sin (\omega_{k,k+1}),\cos(\omega_{k,k+1})).
\end{equation}

Now, we consider the Riemann problem for \eqref{eq-reaction-h} with boundary,
\begin{eqnarray}\label{eq-br}
\left\{
\begin{array}{ll}
W(U)_x+H(U)_y=0,  & \text{in $\Omega_{k,h}$},\\
U|_{\{x=kh\}}=U_a, & \\
(u,v)\cdot \textbf{n}_k=0 & \text{on $\Gamma_{k}$},
\end{array}
\right.
\end{eqnarray}
where $U_a$ is a constant state (see Fig. \ref{4}).

For small angle $\omega_{k,k+1}$, we have the following for the solvability of the boundary Riemann problem \eqref{eq-br}.
\begin{lemma}\label{lemma-R2}
There exists $\epsilon>0$ such that, for $U_a\in O_{\epsilon}({U}_2^{(0)})$ and $|\omega_{k,k+1}|<\epsilon$, there is only one admissible solution,
consisting of a 1-wave with strength $\gamma_1$, that solves the boundary value problem \eqref{eq-br}. It also holds that
\begin{equation}\label{probelm-R2}
\gamma_1=K_b\omega_{k,k+1}+O(1)(|\omega_{k,k+1}|^2+|U_a-{U}_2^{(0)}|),
\end{equation}
with the constant $K_b>0$.

%
%
%
\end{lemma}
\begin{proof}
Let us consider the function
\[
\varphi_k(\gamma_1,\omega_{k,k+1})=(u,v)\cdot\textbf{n}_k=\Big(\Phi_{1}^{(1)}(\gamma_1;U_a),\Phi_{1}^{(2)}(\gamma_1;U_a)\Big)\cdot
\Big(-\sin (\omega_{k,k+1}),\cos(\omega_{k,k+1})\Big),
\]
where $\Phi_1^{(i)}(i=1,2)$ is the $i$-th component of  $\Phi_1$ .

Note that $\varphi_k(0,0)|_{\{U_a={U}_2^{(0)}\}}=0$, and
$$
 \frac{\partial \varphi_k(\gamma_1,\omega_{k,k+1})}{\partial\gamma_1}{\big|_{\{\gamma_1=0,\omega_{k,k+1}=0,U_a
 ={U}_2^{(0)}\}}}=\kappa_1({U}_2^{(0)})>0,
$$
with $\kappa_1({U}_2^{(0)})$ given by Lemma \ref{A.1}. It follows from the implicit function theorem that there exists $\epsilon>0$, such that for $U_a\in O_{\epsilon}({U}_2^{(0)})$ and $|\omega_{k,k+1}|<\epsilon$, equation $\varphi_k(\gamma_1,\omega_{k,k+1})=0$ admits a unique solution  $\gamma_1(\omega_{k,k+1})$. Moreover, by the Taylor expansion formula, we have
\[
\gamma_1(\omega_{k,k+1})=\gamma_1(0)+\frac{\partial\gamma_1}{\partial\omega_{k,k+1}}{\big|_{\{\omega_{k,k+1}=0\}}} \omega_{k,k+1}+O(1)|\omega_{k,k+1}|^2.
\]

Differentiating $\varphi_k(\gamma_1(\omega_{k,k+1}),\omega_{k,k+1})=0$ with respect to $\omega_{k,k+1}$, and letting $\omega_{k,k+1}=0$ and $U_a={U}_2^{(0)}$, we have
\[
\frac{\partial\gamma_1}{\partial\omega_{k,k+1}}{\big|_{\{\omega_{k,k+1}=0,U_a={U}_2^{(0)}\}}}=\frac{{u}_2^{(0)}}{\kappa_1({U}_2^{(0)})}>0.
\]
Thus, we have $K_b>0$ for sufficiently small $\epsilon$.
\end{proof}

 Then, we can obtain the estimates of the weak wave reflection on the boundary.
\begin{proposition}\label{prop-in-b}
Suppose that the three constant states $U_a,U_m$ and $U_{k-1}\in O_{\epsilon}({U}_2^{(0)})$ satisfy that (see Fig. \ref{4})
\begin{align}
&V_m=\tilde{\Phi}(\alpha_5,\alpha_3,\alpha_2;V_a),\quad Z_m=Z_a+\alpha_4,\\
&U_{k-1}=\Phi_1(\beta_1;U_m),\quad (u_{k-1},v_{k-1}) \cdot \textbf{n}_{k-1}=0.
\end{align}
Then, for constant state $U_{k}\in O_{\epsilon}({U}_2^{(0)})$ which satisfies that
\[
U_{k}=\Phi_1(\gamma_1;U_a), \quad (u_{k},v_{k}) \cdot \textbf{n}_{k}=0,
\]
it holds that
\begin{equation}\label{estimate-in-b}
\gamma_1=\beta_1+K_{b0}\omega_{k}+K_{b2}\alpha_2+K_{b3}\alpha_3+K_{b5}\alpha_5,
\end{equation}
where $K_{b0},K_{b2},K_{b3},K_{b5}$ are $C^2$-functions of $\beta_1,\omega_{k},\alpha_2,\alpha_3,\alpha_5,\omega_{k-1,k}$ and $U_a$.  Furthermore, $K_{b0}$ is bounded, and when $\beta_1=\omega_{k}=\alpha_2=\alpha_3=\alpha_5=\omega_{k-1,k}=0,U_a={U}_2^{(0)}$, it holds that
\begin{equation}\label{estimate-Kb}
K_{b5}=1,\quad K_{bi}=0,\quad \text{$i=2,3$}.
\end{equation}
\end{proposition}
\begin{proof}
Let us consider the function:
\begin{align*}
&\varphi_{k,k-1}(\gamma_1,\beta_1,\omega_k,\alpha_2,\alpha_3,\alpha_5)\\
:=&\Big(\Phi_{1}^{(1)}(\gamma_1;V_a),\Phi_{1}^{(2)}(\gamma_1;V_a)\Big)\cdot \textbf{n}_{k}
-\Big(\Phi_{1}^{(1)}(\beta_1;\tilde{\Phi}(\alpha_5,\alpha_3,\alpha_2;V_a)),\Phi_{1}^{(2)}
(\beta_1;\tilde{\Phi}(\alpha_5,\alpha_3,\alpha_2;V_a))\Big)\cdot\textbf{n}_{k-1}.
\end{align*}

Note that $\varphi_{k,k-1}(0,0,0,0,0,0)=0$ and $\frac{\partial\varphi_{k,k-1}}{\partial\gamma_1}|_{\{\gamma_1=0,U_{a}={U}_2^{(0)},\omega_{k,k+1}=0\}}=\kappa_1({U}_2^{(0)})>0$,
it follows from the implicit function theorem that $\gamma_1$ can be solved as a $C^2$ function of $\beta_1,\omega_k,\alpha_2,\alpha_3,\alpha_5,\omega_{k-1,k}$,and $V_a$. Next, by the Taylor expansion formula,  we have
\begin{align*}
\gamma_1&=\gamma_1(\beta_1,0,0,0,0)+\gamma_1(\beta_1,\omega_k,0,0,0)-\gamma_1(\beta_1,0,0,0,0)+\gamma_1(\beta_1,\omega_k,\alpha_2,0,0)\\
&\hspace{1em}-\gamma_1(\beta_1,\omega_k,0,0,0)+\gamma_1(\beta_1,\omega_k,\alpha_2,\alpha_3,0)-\gamma_1(\beta_1,\omega_k,\alpha_2,0,0)\\
&\hspace{1em}+\gamma_1(\beta_1,\omega_k,\alpha_2,\alpha_3,\alpha_5)-\gamma_1(\beta_1,\omega_k,\alpha_2,\alpha_3,0)\\
&=\beta_1+K_{b0}\omega_{k}+K_{b2}\alpha_2+K_{b3}\alpha_3+K_{b5}\alpha_5.
\end{align*}

Differentiating $\varphi_{k,k-1}(\gamma_1,\beta_1,\omega_k,\alpha_2,\alpha_3,\alpha_5)=0$ with respect to $\omega_k,\alpha_2,\alpha_3,\alpha_5$, and letting $\beta_1=\omega_{k}=\alpha_2=\alpha_3=\alpha_5=\omega_{k-1,k}=0,$ and letting $U_a={U}_2^{(0)}$, we have
\[
\frac{\partial\gamma_1}{\partial\omega_k}=\frac{{u}_2^{(0)}}{\kappa_1({U}_2^{(0)})},\quad \frac{\partial\gamma_1}{\partial\alpha_i}=\frac{r_{i}^{(2)}({U}_2^{(0)})}{\kappa_1({U}_2^{(0)})},\quad i=2,3,5,
\]
where $r_{5}^{(2)}({U}_2^{(0)})=\kappa_5({U}_2^{(0)}),r_{2}^{(2)}({U}_2^{(0)})=r_{3}^{(2)}({U}_2^{(0)})=0.$ Then by Lemma \ref{A.1}, we have \eqref{estimate-Kb}.
\end{proof}

Finally, let us consider the wave interaction estimates involving the strong contact discontinuity for \eqref{eq-reaction-h}. First we have the following lemma.
\begin{lemma}\label{estimate-F}
For the constant states ${V}_1^{(0)}=({u}_1^{(0)},0,{p}_1^{(0)},{\rho}_1^{(0)})^\top$ and ${V}_2^{(0)}=({u}_2^{(0)},0,{p}_2^{(0)},{\rho}_2^{(0)})^\top$,  it holds that
\begin{enumerate}
\item
\begin{align}\label{estimate-det}
\det(&\tilde{r}_5({V}_2^{(0)}),\partial_{\sigma_3}\mathcal{F}(\sigma_{30},\sigma_{20};{V}_1^{(0)}),\partial_{\sigma_2}
\mathcal{F}(\sigma_{30},\sigma_{20};{V}_1^{(0)}),
\nabla_{V}\mathcal{F}(\sigma_{30},\sigma_{20};{V}_1^{(0)})\cdot\tilde{r}_1({V}_1^{(0)}))\nonumber\\
&=\kappa_{1}({V}_1^{(0)})\kappa_{5}({V}_2^{(0)})({\rho}_{1}^{(0)})^{2}({u}_{1}^{(0)})^{2}e^{\sigma_{20}+\sigma_{30}}
(\lambda_{5}({V}_2^{(0)})e^{2\sigma_{20}+\sigma_{30}}+\lambda_{5}({V}_1^{(0)}))>0.
\end{align}
where $\tilde{r}_i(i=1,5)$ is the vector which only contains the first four components of $r_i$.
\item
For any $V_a\in O_{\epsilon}({V}_1^{(0)})$ and $\sigma_j\in O_{\hat{\epsilon}}(\sigma_{j0})$ which satisfies that $\mathcal{F}(\sigma_3,\sigma_2;V_a)\in O_{\epsilon}({V}_2^{(0)})$ with some $\hat{\epsilon}=\hat{\epsilon}(\epsilon)\to 0$ as $\epsilon\to 0$,
 it holds that
\begin{equation}\label{diff-estimate-f}
|\mathcal{F}(\sigma_3,\sigma_2;V_a)-\mathcal{F}(\sigma_{30},\sigma_{20};V_a)|\leq C(|\sigma_3-\sigma_{30}|+|\sigma_2-\sigma_{20}|),
\end{equation}
for some constant $C$.
\end{enumerate}
\end{lemma}
\begin{proof}
Since $\mathcal{F}(\sigma_3,\sigma_2;V_a)
=(u_a e^{\sigma_2},v_a e^{\sigma_2},p_a,\rho_a e^{\sigma_3})^\top$ for any  $V_a\in O_{\epsilon}({V_1}^{(0)})$, ${u}_2^{(0)}={u}_1^{(0)} e^{\sigma_{20}}$, and ${\rho}_2^{(0)}={\rho}_1^{(0)} e^{\sigma_{30}}$, direct calculations gives that,
\begin{align*}
&\det(\tilde{r}_5({V}_2^{(0)}),\partial_{\sigma_3}\mathcal{F}(\sigma_{30},\sigma_{20};{V}_1^{(0)}),\partial_{\sigma_2}
\mathcal{F}(\sigma_{30},\sigma_{20};{V}_1^{(0)}),
\nabla_{V}\mathcal{F}(\sigma_{30},\sigma_{20};{V}_1^{(0)})\cdot\tilde{r}_1({V}_1^{(0)}))\\
&=\kappa_{1}({V}_1^{(0)})\kappa_{5}({V}_2^{(0)})\begin{vmatrix}
-\lambda_5({V}_2^{(0)}) & 0 & {u}_1^{(0)} e^{\sigma_{20}} &-\lambda_1({V}_1^{(0)})e^{\sigma_{20}}\\
1& 0 & 0 & e^{\sigma_{20}} \\
\lambda_5({V}_2^{(0)}){\rho}_{2}^{(0)}{u}_{2}^{(0)} & 0 & 0 & \lambda_1({V}_1^{(0)}){\rho}_{1}^{(0)}{u}_{1}^{(0)}\\
\lambda_5({V}_2^{(0)}){\rho}_{2}^{(0)}{u}_{2}^{(0)}/{({c}_2^{(0)})}^2 & {\rho}_1^{(0)} e^{\sigma_{30}} &0 &\lambda_1({V}_2^{(0)}){\rho}_{1}^{(0)}{u}_{1}^{(0)}e^{\sigma_{30}}/{({c}_1^{(0)})}^2
\end{vmatrix}\\
&=\kappa_{1}({V}_1^{(0)})\kappa_{5}({V}_2^{(0)})({\rho}_{1}^{(0)})^{2}({u}_{1}^{(0)})^{2}e^{\sigma_{20}+\sigma_{30}}
(\lambda_{5}({V}_2^{(0)})e^{2\sigma_{20}+\sigma_{30}}+\lambda_{5}({V}_1^{(0)}))>0.
\end{align*}

Moreover, note that
$$
\mathcal{F}(\sigma_3,\sigma_2;V_a)-\mathcal{F}(\sigma_{30},\sigma_{20};V_a)=(u_a(e^{\sigma_2}-e^{\sigma_{20}}),v_a (e^{\sigma_2}-e^{\sigma_{20}}),0,\rho_a(e^{\sigma_3}-e^{\sigma_{30}}))^\top,
$$
then by the Taylor expansion formula, we can obtain \eqref{diff-estimate-f}. 
\end{proof}

We remark that \eqref{estimate-det} is essential to estimate the strengths of reflected weak waves in the wave interaction of the strong contact discontinuity and weak waves governed by \eqref{eq-reaction-h}. Now, we can establish the solvability of the Riemann problem involving the strong contact discontinuity.
\begin{lemma}\label{lemma-R3}
There exists $\epsilon>0$ such that, for any given constant states $U_a\in O_{\epsilon}({U}_1^{(0)})$ and $U_b\in O_{\epsilon}({U}_2^{(0)})$, the Riemann problem \eqref{initial-r} admits a unique admissible solution that consists of a weak 1-wave, a strong contact discontinuity, and a weak 5-wave. In addition, $U_b$ can be represented by
\begin{equation}\label{problem-R3}
\begin{cases}
&V_b=\tilde{\Phi}_5(\alpha_5;\mathcal{F}(\sigma_3,\sigma_2;\tilde{\Phi}_1(\alpha_1;V_a))),\\
&Z_b=Z_a+\alpha_4,
\end{cases}
\end{equation}
with $V_b=(u_b,v_b, p_b, \rho_b)^\top$.
\end{lemma}
\begin{proof}
It is clear from  \eqref{c3-1} that $Z_b=Z_a+\alpha_4$.

Next, let us consider the function:
\[
\varphi_c(\alpha_5,\sigma_3,\sigma_2,\alpha_1,{V}_a,V_b)=
\tilde{\Phi}_5(\alpha_5;\mathcal{F}(\sigma_3,\sigma_2;\tilde{\Phi}_1(\alpha_1;{V}_a)))-V_b.
\]
Obviously, we have $\varphi_c(0,\sigma_{30},\sigma_{20},0,{V}_1^{(0)},{V}_2^{(0)})=0$, and
\begin{align*}
&\det\Big(\frac{\partial\varphi_c(\alpha_5,\sigma_3,\sigma_2,\alpha_1,{V}_a,V_b)}
{\partial(\alpha_5,\sigma_3,\sigma_2,\alpha_1)}\Big)
_{|_{\{\alpha_5=\alpha_1=0,\sigma_3=\sigma_{30},\sigma_2=\sigma_{20},V_a={V}_1^{(0)},V_b={V}_2^{(0)}\}}}\nonumber\\
&=\det(\tilde{r}_5({V}_2^{(0)}),\partial_{\sigma_3}\mathcal{F}(\sigma_{30},\sigma_{20};{V}_1^{(0)}),
\partial_{\sigma_2}\mathcal{F}(\sigma_{30},\sigma_{20};{V}_1^{(0)}),\tilde{r}_1({V}_1^{(0)}))\\
&=\kappa_{1}({V}_1^{(0)})\kappa_{5}({V}_2^{(0)})\begin{vmatrix}
-\lambda_5({V}_2^{(0)}) & 0 & {u}_1^{(0)} e^{\sigma_{20}} &-\lambda_1({V}_1^{(0)})\\
1& 0 & 0 & 1 \\
\lambda_5({V}_2^{(0)}){\rho}_{2}^{(0)}{u}_{2}^{(0)} & 0 & 0 & \lambda_1({V}_1^{(0)}){\rho}_{1}^{(0)}{u}_{1}^{(0)}\\
\lambda_5({V}_2^{(0)}){\rho}_{2}^{(0)}{u}_{2}^{(0)}/{({c}_2^{(0)})}^2 & {\rho}_1^{(0)} e^{\sigma_{30}} &0 &\lambda_1({V}_2^{(0)}){\rho}_{1}^{(0)}{u}_{1}^{(0)}/{({c}_1^{(0)})}^2
\end{vmatrix}\\
&=\kappa_{1}({V}_1^{(0)})\kappa_{5}({V}_2^{(0)}){(\rho}_{1}^{(0)})^{2}({u}_{1}^{(0)})^{2}e^{\sigma_{20}+\sigma_{30}}
(\lambda_{5}({V}_2^{(0)})e^{\sigma_{20}+\sigma_{30}}+\lambda_{5}({V}_1^{(0)}))>0.
\end{align*}

Then it follows from the implicit function theorem that there exists $\epsilon>0$, such that for any given constant states $U_a\in O_{\epsilon}({U}_1^{(0)})$ and $U_b\in O_{\epsilon}({U}_2^{(0)})$, the equation
$$
\varphi_c(\alpha_5,\sigma_3,\sigma_2,\alpha_1,V_a,V_b)=0
$$
admits a unique solution $\alpha_5,\sigma_3,\sigma_2,\alpha_1$.
\end{proof}

Now we shall derive the wave interaction estimates between the strong contact discontinuity and weak waves. There are two cases depending on how the strong contact discontinuity and weak waves interact. The first case is that, as shown in Fig. \ref{5}, the weak waves approach the strong contact discontinuity from the above.  For this case, we have  the following lemma.
\begin{figure}[ht]
\begin{center}
\setlength{\unitlength}{1mm}
\begin{picture}
(50,70)(1,2)
\linethickness{1pt}
\put(0,0){\line(0,1){72}}
\put(26,0){\line(0,1){72}}
\put(52,0){\line(0,1){72}}
\put(0,54){\line(3,2){20}}
\put(0,54){\line(3,-1){20}}
\put(0,54){\line(5,2){20}}
\put(0,18){\line(4,3){20}}
\thicklines
\put(0,18){\line(5,2){20}}
\thinlines
\put(0,18){\line(4,-1){20}}
\put(26,40){\line(2,-1){20}}
\put(26,40){\line(3,2){20}}
\thicklines
\put(26,40){\line(5,1){20}}
\thinlines
\put(2,68){$U_b$}
\put(2,35){$U_m$}
\put(2,5){$U_a$}
\put(36,18){$U_a$}
\put(36,60){$U_b$}
\put(12,66){$\beta_5$}
\put(12,56){$\beta_{2(3,4)}$}
\put(12,46){$\beta_1$}
\put(12,34){$\alpha_5$}
\put(12,20){$\sigma_{2(3)},\alpha_{4}$}
\put(12,11){$\alpha_1$}
\put(40,54){$\gamma_5$}
\put(38,39){$\sigma'_{2(3)},\gamma_{4}$}
\put(40,29){$\gamma_1$}
\end{picture}
\end{center}
\caption{Weak waves approach the strong contact discontinuity from above.}\label{5}
\end{figure}
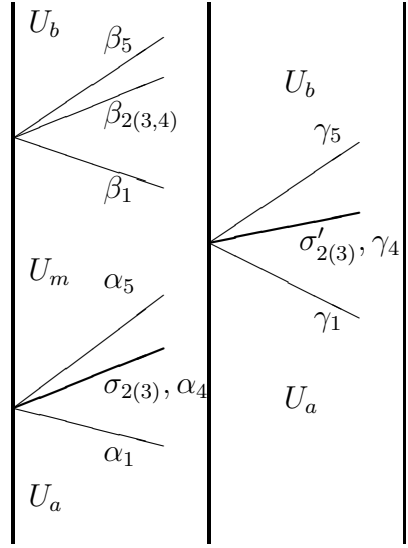

\begin{proposition}\label{prop-in-c1}
For any given three constant states $U_a\in O_{\epsilon}({U}_1^{(0)})$, and $U_m,U_b\in O_{\epsilon}({U}_2^{(0)})$, (see Fig. \ref{5}), with the assumptions that
\begin{align*}
&V_m=\tilde{\Phi}_5(\alpha_5;\mathcal{F}(\sigma_3,\sigma_2;\tilde{\Phi}_1(\alpha_1;V_a))),\quad Z_m=Z_a+\alpha_4,\\
&V_b=\tilde{\Phi}(\beta_5,\beta_3,\beta_2,\beta_1;V_m),\quad Z_b=Z_m+\beta_4,\\
&V_b=\tilde{\Phi}_5(\gamma_5;\mathcal{F}(\sigma'_3,\sigma'_2;\tilde{\Phi}_1(\gamma_1;V_a))),\quad Z_b=Z_a+\gamma_4,
\end{align*}
it holds that
\begin{equation}\label{estimate-in-c1}
\begin{cases}
&\gamma_1=K_{21}\beta_1+\alpha_1+O(1)\Delta'(\alpha_{5},\boldsymbol{\beta}^{*}),\\
&\sigma'_{i}=K_{2i}\beta_1+\beta_i+\sigma_i+O(1)\Delta'(\alpha_{5},\boldsymbol{\beta}^{*}),i=2,3,\\
&\gamma_4=\alpha_4+\beta_4,\\
&\gamma_5=K_{25}\beta_1+\alpha_5+\beta_5+O(1)\Delta'(\alpha_{5},\boldsymbol{\beta}^{*}),
\end{cases}
\end{equation}
where $\Delta'(\alpha_{5},\boldsymbol{\beta}^{*})=|\alpha_5|(|\beta_1|+|\beta_2|+|\beta_3|)+\Delta_5(\alpha_{5},\beta_{5})$. Furthermore, $\sum\limits_{i=1}^{3}|K_{2i}|$ is bounded, and when $\beta=\alpha_1=\alpha_4=\alpha_5=0,\sigma_2=\sigma_{20},\sigma_3=\sigma_{30}$, it holds that $|K_{25}|<1$.
\end{proposition}
\begin{remark}
The essential feature of homogeneous system \eqref{eq-reaction-h} is that the reflection coefficient $K_{25}$ is less than one, which is the stability condition in $\textup{\cite{Tougeron1997,Tougeron1993}}$.
\end{remark}

\begin{proof}
First, it is obvious that $\gamma_4=\alpha_4+\beta_4.$

Then, for any state $V_{\widetilde{m}}\in O_{\epsilon}({V}_2^{(0)})$, we define
\begin{equation}\label{function-delta}
\tilde{\Phi}(\delta_5,\delta_3,\delta_2,\delta_1;V_{\widetilde{m}})=
\tilde{\Phi}(\beta_5,\beta_3,\beta_2,\beta_1;\tilde{\Phi}_5(\alpha_5;V_{\widetilde{m}})).
\end{equation}
By applying Proposition \ref{prop-in-w}, we have
\begin{align}\label{estimate-delta}
&\delta_i=\beta_i+O(1)\Delta'(\alpha_{5},\boldsymbol{\beta}^{*}), \quad i=1,2,3,\nonumber\\
&\delta_5=\alpha_5+\beta_5+O(1)\Delta'(\alpha_{5},\boldsymbol{\beta}^{*}),
\end{align}
where $\Delta'(\alpha_{5},\boldsymbol{\beta}^{*})=|\alpha_5|(|\beta_1|+|\beta_2|+|\beta_3|)+\Delta_5(\alpha_{5},\beta_{5})$.

Let $\boldsymbol{\delta}^{*}=(\delta_5,\delta_3,\delta_2,\delta_1)$.  By \eqref{function-delta}, let us consider the following function:
\begin{align*}
&\varphi_d(\gamma_5,\sigma'_{3},\sigma'_{2},\gamma_1,\boldsymbol{\delta}^{*},\sigma_{3},\sigma_{2},\alpha_1)\\
=&\tilde{\Phi}_5(\gamma_5;\mathcal{F}(\sigma'_3,\sigma'_2;\tilde{\Phi}_1(\gamma_1;V_a)))
-\tilde{\Phi}(\beta_5,\beta_3,\beta_2,\beta_1;\tilde{\Phi}_5(\alpha_5;\mathcal{F}(\sigma_3,\sigma_2;\tilde{\Phi}_1(\alpha_1;V_a))))\\
=&\tilde{\Phi}_5(\gamma_5;\mathcal{F}(\sigma'_3,\sigma'_2;\tilde{\Phi}_1(\gamma_1;V_a)))
-\tilde{\Phi}(\delta_5,\delta_3,\delta_2,\delta_1;\mathcal{F}(\sigma_3,\sigma_2;\tilde{\Phi}_1(\alpha_1;V_a))).
\end{align*}

It is clear that $\varphi_d(0,\sigma_{30},\sigma_{20},0,\boldsymbol{0},\sigma_{30},\sigma_{20},0)=0$. By \eqref{estimate-det}, we have
\begin{align*}
&\det\Big(\frac{\partial\varphi_d(\gamma_5,\sigma'_{3},\sigma'_{2},\gamma_1,\boldsymbol{\delta}^{*},\sigma_{3},\sigma_{2},\alpha_1)}
{\partial(\gamma_5,\sigma'_{3},\sigma'_{2},\gamma_1)}\Big)|_{\{\gamma_1=\gamma_5=0,\sigma'_{3}=\sigma_{30},\sigma'_{2}=
\sigma_{20},V_a={V}_1^{(0)}\}}\\
&=\det(\tilde{r}_5({V}_2^{(0)}),\partial_{\sigma_3}\mathcal{F}(\sigma_{30},\sigma_{20};{V}_1^{(0)}),\partial_{\sigma_2}
\mathcal{F}(\sigma_{30},\sigma_{20};{V}_1^{(0)}),\nabla_{V}\mathcal{F}(\sigma_{30},\sigma_{20};{V}_1^{(0)})\cdot\tilde{r}_1({V}_1^{(0)}))>0.
\end{align*}
 Then it follows from the implicit function theorem that $\gamma_i,i=1,5,$ and $\sigma'_{j}, j=2,3,$ can be solved as a $C^2$ function of $\gamma_5,\sigma'_{3},\sigma'_{2},\gamma_1,\boldsymbol{\delta}^{*},\sigma_{3},\sigma_{2},\alpha_1$, and $V_a$. Thus,
we have
\begin{align*}
\sigma'_j&=\sigma'_j(\delta_5,\delta_3,\delta_2,\delta_1,\sigma_{3},\sigma_{2},\alpha_1)
-\sigma'_j(\delta_5,\delta_3,\delta_2,0,\sigma_{3},\sigma_{2},\alpha_1)+\sigma'_j(\delta_5,\delta_3,\delta_2,0,\sigma_{3},\sigma_{2},\alpha_1)\\
&=K_{2j}\delta_1+\delta_j+\sigma_j,\quad j=2,3,
\end{align*}
Similarly, it holds that
$$
\gamma_1=K_{21}\delta_1+\alpha_1,\qquad\mbox{ and }\qquad\gamma_5=K_{25}\delta_1+\delta_5.
$$
Then by \eqref{estimate-delta}, we can obtain \eqref{estimate-in-c1}.

Differentiating the equation $\varphi_d=0$ with respect to $\delta_1$, and letting $\boldsymbol{\delta}^{*}=\alpha_1=0,\sigma_{3}=\sigma_{30},\sigma_{2}=\sigma_{20},$ and $U_a={U}_1^{(0)}$, we have
\begin{align*}
\partial_{\delta_1}\gamma_5\tilde{r}_5({V}_2^{(0)})&+\partial_{\delta_1}\sigma'_{3}\partial_{\sigma_3}\mathcal{F}(\sigma_{30},\sigma_{20};{V}_1^{(0)})
+\partial_{\delta_1}\sigma'_{2}\partial_{\sigma_2}\mathcal{F}(\sigma_{30},\sigma_{20};{V}_1^{(0)})\\
&+\partial_{\delta_1}\gamma_1\nabla_{V}\mathcal{F}(\sigma_{30},\sigma_{20};{V}_1^{(0)})\cdot\tilde{r}_1({V}_1^{(0)})=\tilde{r}_1({V}_2^{(0)}).
\end{align*}
It is clear that $K_{2i},i=1,2,3$ are bounded. By \eqref{estimate-det} and Lemma \ref{A.1}, it holds that
\begin{align*}
|\partial_{\delta_1}\gamma_5|
&=\Bigg|\frac{\det(\tilde{r}_1({V}_2^{(0)}),\partial_{\sigma_3}\mathcal{F}(\sigma_{30},\sigma_{20};{V}_1^{(0)}),\partial_{\sigma_2}
\mathcal{F}(\sigma_{30},\sigma_{20};{V}_1^{(0)}),\nabla_{V}\mathcal{F}(\sigma_{30},\sigma_{20};{V}_1^{(0)})\cdot\tilde{r}_1({V}_1^{(0)}))}
{\det(\tilde{r}_5({V}_2^{(0)}),\partial_{\sigma_3}\mathcal{F}(\sigma_{30},\sigma_{20};{V}_1^{(0)}),
\partial_{\sigma_2}\mathcal{F}(\sigma_{30},\sigma_{20};{V}_1^{(0)}),
\nabla_{V}\mathcal{F}(\sigma_{30},\sigma_{20};{V}_1^{(0)})\cdot\tilde{r}_1({V}_1^{(0)}))}\Bigg|\\
&=\Bigg|\frac{\kappa_{1}({V}_1^{(0)})\kappa_{1}({V}_2^{(0)})({\rho}_{1}^{(0)})^{2}({u}_{1}^{(0)})^{2}e^{\sigma_{20}+\sigma_{30}}(\lambda_{5}({V}_1^{(0)})
-\lambda_{5}({V}_2^{(0)})e^{2\sigma_{20}+\sigma_{30}})}
{\kappa_{1}({V}_1^{(0)})\kappa_{5}({V}_2^{(0)})({\rho}_{1}^{(0)})^{2}({u}_{1}^{(0)})^{2}e^{\sigma_{20}+\sigma_{30}}(\lambda_{5}({V}_1^{(0)})+
\lambda_{5}({V}_2^{(0)})e^{2\sigma_{20}+\sigma_{30}})}\Bigg|\\
&=\Bigg|\frac{\lambda_{5}({V}_1^{(0)})-
\lambda_{5}({V}_2^{(0)})e^{2\sigma_{20}+\sigma_{30}}}{\lambda_{5}({V}_1^{(0)})+
\lambda_{5}({V}_2^{(0)})e^{2\sigma_{20}+\sigma_{30}}}\Bigg|<1.
\end{align*}
This completes the proof.
\end{proof}

The second case is that the weak waves approach the strong contact discontinuity from the below (Fig. \ref{6}). By the symmetry, we can easily obtain the following proposition.
\begin{proposition}\label{prop-in-c2}
For any given three constant states $U_a,U_m\in O_{\epsilon}({U}_1^{(0)})$, and $U_b\in O_{\epsilon}({U}_2^{(0)})$ with the assumptions that
\begin{align*}
&V_m=\tilde{\Phi}(\alpha_5,\alpha_3,\alpha_2,\alpha_1;V_a))),\quad Z_m=Z_a+\alpha_4,\\
&V_b=\tilde{\Phi}_5(\beta_5;\mathcal{F}(\sigma_3,\sigma_2;\tilde{\Phi}_1(\beta_1;V_m))),\quad Z_b=Z_m+\beta_4,\\
&V_b=\tilde{\Phi}_5(\gamma_5;\mathcal{F}(\sigma'_3,\sigma'_2;\tilde{\Phi}_1(\gamma_1;V_a))),\quad Z_b=Z_a+\gamma_4,
\end{align*}
it holds that
\begin{equation}\label{estimate-in-c2}
\begin{cases}
&\gamma_1=K_{11}\alpha_5+\alpha_1+\beta_1+O(1)\Delta''(\boldsymbol{\alpha}^{*},\beta_{1}),\\
&\sigma'_{i}=K_{1i}\alpha_5+\alpha_i+\sigma_i+O(1)\Delta''(\boldsymbol{\alpha}^{*},\beta_{1}),i=2,3,\\
&\gamma_4=\alpha_4+\beta_4,\\
&\gamma_5=K_{15}\alpha_5+\beta_5+O(1)\Delta''(\boldsymbol{\alpha}^{*},\beta_{1}),
\end{cases}
\end{equation}
where $\Delta''(\boldsymbol{\alpha}^{*},\beta_{1})
=|\beta_1|(|\alpha_5|+|\alpha_3|+|\alpha_2|)+\Delta_1(\alpha_{1},\beta_{1})$.
\end{proposition}
\begin{figure}[ht]
\begin{center}
\setlength{\unitlength}{1mm}
\begin{picture}(50,68)(1,2)
\linethickness{1pt}
\put(0,0){\line(0,1){72}}
\put(26,0){\line(0,1){72}}
\put(52,0){\line(0,1){72}}
\put(0,54){\line(3,2){20}}
\put(0,54){\line(3,-1){20}}
\thicklines
\put(0,54){\line(5,2){20}}
\thinlines
\put(0,18){\line(4,3){20}}
\put(0,18){\line(5,2){20}}
\put(0,18){\line(4,-1){20}}
\put(26,40){\line(2,-1){20}}
\put(26,40){\line(3,2){20}}
\thicklines
\put(26,40){\line(5,1){20}}
\thinlines
\put(2,68){$U_b$}
\put(2,35){$U_m$}
\put(2,5){$U_a$}
\put(36,18){$U_a$}
\put(36,60){$U_b$}
\put(12,66){$\beta_5$}
\put(12,56){$\sigma_{2(3)},\beta_{4}$}
\put(12,46){$\beta_1$}
\put(12,34){$\alpha_5$}
\put(12,20){$\alpha_{2(3,4)}$}
\put(12,11){$\alpha_1$}
\put(40,54){$\gamma_5$}
\put(38,39){$\sigma'_{2(3)},\gamma_{4}$}
\put(40,29){$\gamma_1$}
\end{picture}
\end{center}
\caption{Weak waves approach the strong contact discontinuity from below.}\label{6}
\end{figure}

\subsection{Local estimates on the reacting step}
Let $\tilde{U}=(\tilde{u},\tilde{v},\tilde{p},\tilde{\rho},\tilde{Z})^\top$ be the value of $U$ after the reaction. It means that $\tilde{U}$ satisfies the equation
\[
W(\tilde{U})=W(U)+G(U)h,
\]
which is precisely of the following form
\begin{equation}\label{equation-Re}
\begin{cases}
&\tilde{\rho}\tilde{u}=\rho u,\\
&\tilde{\rho}\tilde{u}^2+\tilde{p}=\rho u^2+p,\\
&\tilde{\rho}\tilde{u}\tilde{v}=\rho uv,\\
&(\tilde{\rho}\tilde{E}+\tilde{p})\tilde{u}=(\rho E+p)u+q_0\rho\phi(T)Zh,\\
&\tilde{\rho}\tilde{u}\tilde{Z}=\rho u Z-\rho\phi(T)Zh.
\end{cases}
\end{equation}

Then we have the following property that indicates the change of the solutions $\tilde{U}$ with respect to $h$.
\begin{lemma}\label{lemma-Re}
Suppose that $0\leq Z\leq 1$ and $T \ge T_0$ for some positive constant $T_0$, then there exists a constant $l>0$, such that for sufficiently small $h>0$, it holds that
\[
\tilde{T}\ge T \ge T_0>0,\quad \tilde{V}-V=O(1)Zh,\quad 0\leq\tilde{Z}\le e^{-lh}Z\leq 1,
\]
where $\tilde{V}=(\tilde{u},\tilde{v},\tilde{p},\tilde{\rho})^\top$, and $V=(u,v,p,\rho)^\top$.
\end{lemma}
\begin{proof}
By $\eqref{equation-Re}_1$ and $\eqref{equation-Re}_2$, we have that
\begin{equation}
\tilde{u}-u=-\frac{1}{\rho u}(\tilde{p}-p).
\end{equation}
By $\eqref{equation-Re}_1$ and $\eqref{equation-Re}_3$, we have that
\begin{equation}
\tilde{v}=v.
\end{equation}
By $\eqref{equation-Re}_1$ and $\eqref{equation-Re}_5$, we know that
\begin{equation}\label{eq-Z}
\tilde{Z}=(1-\frac {\phi(T)}{u}h)Z.
\end{equation}
Moreover, $\eqref{equation-Re}_1$ also means that
\begin{equation}
\tilde{\rho}-\rho=-\frac{\rho}{u}(\tilde{u}-u)+O(h^2).
\end{equation}

Note that by the thermodynamical relation, we know that $T=\frac{\gamma-1}{R}e=\frac{p}{R\rho}$.
Then by the assumption $u^2>c^2=\gamma RT$ and from all the above identities and $\eqref{equation-Re}_4$, we have that
\[
\tilde{T}-T=\frac{(\gamma -1)(u^2-RT)}{R\rho u(u^2-\gamma RT)}q_0\rho\phi(T)Zh+O(h^2)\geq 0,
\]
which shows that the temperature $T$ does not decrease due to the reaction.

Next, \eqref{eq-Z} also means that $0\leq\tilde{Z}\leq 1.$ Since $\phi(T)$ is assumed to be Lipschitz continuous, nonnegative, and increasing, there exists a constant $l > 0$, such that $\tilde{Z}\leq e^{-lh}Z,$ which implies the decay property of the reactant $Z$.

Finally, we will consider the change of $V=(u,v,p,\rho)^\top$. It follows from the implicit function theorem that $\tilde{V}=(\tilde{u},\tilde{v},\tilde{p},\tilde{\rho})^\top$ can be solved as a $C^2$ function of $V$ and $Zh$ by the first four equations of \eqref{equation-Re}. By the Taylor expansion, one can easily see that there exists a function $\widetilde{\mathcal{V}}$ such that
\begin{equation}\label{eq-V}
\tilde{V}=V+\widetilde{\mathcal{V}}(V,Zh)Zh.
\end{equation}
This completes the proof.
\end{proof}

In sequel, if $\tilde{V}$ and $V$ satisfy (\ref{eq-V}), and if $\tilde{Z}$ and $Z$ satisfy (\ref{eq-Z}), then we say $\tilde{U}=(\tilde{V},\tilde{Z})$ is the value of $U=(V,Z)$ after the reaction step.

Now, we are going to consider the change of the wave strength after the reaction step. The analysis is divided into the following three cases.

\smallskip
{\bf Case 1.} $U_a$ and $U_b$ are connected only by the weak waves.
\begin{proposition}\label{prop-Re1}
Let $U_a,U_b\in O_{\epsilon}({U}_k^{(0)}),k=1,2$ with
\begin{align*}
V_b=\tilde{\Phi}(\gamma_5,\gamma_3,\gamma_2,\gamma_1;V_a),\quad Z_b=Z_a+\gamma_4.
\end{align*}
Let $\tilde{U}_a=(\tilde{V}_a,\tilde{Z}_a)$ and $\tilde{U}_b=(\tilde{V}_b,\tilde{Z}_b)$ be the value of $U_a$ and $U_b$ after the reaction step respectively. Assume that
\begin{align*}
\tilde{V}_b=\tilde{\Phi}(\tilde{\gamma}_5,\tilde{\gamma}_3,\tilde{\gamma}_2,\tilde{\gamma}_1;\tilde{V}_a),\quad \tilde{Z}_b=\tilde{Z}_a+\tilde{\gamma}_4.
\end{align*}
Then it holds that
\begin{equation}\label{estimate-Re1}
\begin{cases}
&\tilde{\gamma}_i=\gamma_i+O(1)|\boldsymbol{\gamma}^{*}|Z_{a}h+O(1)|\gamma_4|h,\quad i=1,2,3,5,\\
&\tilde{\gamma}_4=(1-{\phi(T_b)h}/{u_b})\gamma_4+O(1)|\boldsymbol{\gamma}^{*}|Z_{a}h.
\end{cases}
\end{equation}
where $|\boldsymbol{\gamma}^{*}|=|\gamma_1|+|\gamma_2|+|\gamma_3|+|\gamma_5|$.
\end{proposition}
\begin{proof}
By \eqref{eq-Z}, it is obvious that
$
\tilde{Z}_b=(1-{\phi(T_b)}h/{u_b})Z_b,$ and $ \tilde{Z}_a=(1-{\phi(T_a)}h/{u_a})Z_a.
$
Hence we have
$
\tilde{\gamma}_4=(1-{\phi(T_b)h}/{u_b})\gamma_4+({\phi(T_a)}/{u_a}-{\phi(T_b)}/{u_b})Z_{a}h,
$
which implies $\eqref{estimate-Re1}_2$.

Next, by \eqref{eq-V}, we need to find the solution $\tilde{\boldsymbol{\gamma}}^{*}$ as a function of $\boldsymbol{\gamma}^{*},Z_{a}h,Z_{b}h,$ and $V_a$ such that
\[
\tilde{\Phi}(\tilde{\boldsymbol{\gamma}}^{*};V_a+\widetilde{\mathcal{V}}(V_a,Z_{a}h)Z_{a}h)
=\tilde{\Phi}(\boldsymbol{\gamma}^{*};V_a)
+\widetilde{\mathcal{V}}(V_b,Z_{b}h)Z_{b}h,
\]
where $\tilde{\boldsymbol{\gamma}}^{*}=(\tilde{\gamma}_5,\tilde{\gamma}_3,\tilde{\gamma}_2,\tilde{\gamma}_1)$, and $\boldsymbol{\gamma}^{*}=(\gamma_5,\gamma_3,\gamma_2,\gamma_1)$.

First, it follows from the implicit function theorem that $\tilde{\gamma}_{i},i=1,2,3,5$ can be solved as a $C^2$-function of $(\boldsymbol{\gamma}^{*},Z_{a}h,Z_{b}h, V_a)$ uniquely. Then, we have\begin{align*}
\tilde{\gamma}_{i}(\boldsymbol{\gamma}^{*},Z_{a}h,Z_{b}h,V_a)&=O(1)|Z_a-Z_b|h+\tilde{\gamma}_{i}
(\boldsymbol{\gamma}^{*},Z_{a}h,Z_{a}h,V_a)\\
&=O(1)|Z_a-Z_b|h+O(1)|\boldsymbol{\gamma}^{*}|Z_{a}h+\tilde{\gamma}_{i}(\boldsymbol{\gamma}^{*},0,0,V_a)\\
&\hspace{1em}+\tilde{\gamma}_{i}(0,Z_{a}h,Z_{a}h,V_a)-\tilde{\gamma}_{i}(0,0,0,V_a)\\
&=\gamma_i+O(1)|\boldsymbol{\gamma}^{*}|Z_{a}h+O(1)|\gamma_4|h.
\end{align*}
It completes the proof of this proposition.
\end{proof}

\smallskip
{\bf Case 2.} $U_a$ and $U_k$ are connected by a weak 1-wave near the boundary $\Gamma_k$.

\begin{proposition}\label{prop-Re2}
Let $U_a,U_k\in O_{\epsilon}({U}_2^{(0)})$ with
\begin{align*}
V_{k}=\Phi_1(\gamma_1;V_a), \quad (u_{k},v_{k}) \cdot \textbf{n}_{k}=0.
\end{align*}
Let $\tilde{U}_a=(\tilde{V}_a,\tilde{Z}_a)$ and $\tilde{U}_k=(\tilde{V}_k,\tilde{Z}_k)$ be the value of $U_a,U_k$ after the reaction step respectively. Assume that
\begin{align*}
\tilde{V}_{k}=\Phi_1(\tilde{\gamma}_1;\tilde{V}_a), \quad (\tilde{u}_{k},\tilde{v}_{k}) \cdot \textbf{n}_{k}=0.
\end{align*}
Then, it holds that
\begin{equation}\label{estimate-Re2}
\tilde{\gamma}_1=\gamma_1+O(1)Z_{a}h.
\end{equation}
\end{proposition}
\begin{proof}
By \eqref{eq-V}, we need to find the solution $\tilde{\gamma}_1$ as a function of $\gamma_1,Z_{a}h,$ and $V_a$ such that
\begin{align*}
&(\Phi_1^{(1)}(\tilde{\gamma}_1;V_a+\widetilde{\mathcal{V}}(V_a,Z_{a}h)Z_{a}h),
\Phi_1^{(2)}(\tilde{\gamma}_1;V_a+\widetilde{\mathcal{V}}(V_a,Z_{a}h)Z_{a}h)) \cdot \textbf{n}_{k}\\
&=(\Phi_1^{(1)}(\gamma_1;V_a),\Phi_1^{(2)}(\gamma_1;V_a)) \cdot \textbf{n}_{k}.
\end{align*}

Obviously, it follows from the implicit function theorem that $\tilde{\gamma}_{1}$ can be solved as a $C^2$-function of $(\gamma_1, Z_{a}h, V_a)$ uniquely. Moreover, by the Taylor expansion formula, we have that
\begin{align*}
\tilde{\gamma}_{1}(\gamma_1, Z_{a}h, V_a)&=\tilde{\gamma}_{1}(\gamma_1, 0, V_a)+\tilde{\gamma}_{1}(\gamma_1, Z_{a}h, V_a)-\tilde{\gamma}_{1}(\gamma_1, 0, V_a)=\gamma_1+O(1)Z_{a}h.
\end{align*}
\end{proof}

\smallskip
{\bf Case 3.} $U_a$ and $U_b$ are connected by a weak 1-wave, a strong contact discontinuity, and a weak 5-wave.

\begin{proposition}\label{prop-Re3}
Let $U_a\in O_{\epsilon}({U}_1^{(0)}),U_b\in O_{\epsilon}({U}_2^{(0)})$ with
\begin{align*}
&V_b=\tilde{\Phi}_5(\gamma_5,\mathcal{F}(\sigma_3,\sigma_2;\tilde{\Phi}_1(\gamma_1;V_a))),\quad Z_b=Z_a+\gamma_4.
\end{align*}
Let $\tilde{U}_a$ and $\tilde{U}_b$ be the value of $U_a$ and $U_b$ after the reaction step respectively. Assume that
\begin{align*}
&\tilde{V}_b=\tilde{\Phi}_5(\tilde{\gamma}_5,\mathcal{F}(\tilde{\sigma}_3,\tilde{\sigma}_2;
\tilde{\Phi}_1(\tilde{\gamma}_1;\tilde{V}_a))),\quad \tilde{Z}_b=\tilde{Z}_a+\tilde{\gamma}_4.
\end{align*}
Then, it holds that
\begin{equation}\label{estimate-Re3}
\begin{cases}
&\tilde{\gamma}_i=\gamma_i+O(1)Z_{a}h+O(1)|\gamma_4|h,\,\,i=1,5,\\
&\tilde{\sigma}_j=\sigma_j+O(1)Z_{a}h+O(1)|\gamma_4|h,\,\,j=2,3,\\
&\tilde{\gamma}_4=(1-{\phi(T_b)h}/{u_b})\gamma_4+O(1)Z_{a}h.
\end{cases}
\end{equation}
\end{proposition}
\begin{proof}
By \eqref{eq-Z}, it is clear that
$
\tilde{Z}_b=(1-{\phi(T_b)}h/{u_b})Z_b,$ and $ \tilde{Z}_a=(1-{\phi(T_a)}h/{u_a})Z_a.
$
Hence we have
$
\tilde{\gamma}_4=(1-{\phi(T_b)h}/{u_b})\gamma_4+({\phi(T_a)}/{u_a}-{\phi(T_b)}/{u_b})Z_{a}h,
$
which implies $\eqref{estimate-Re3}_3$.

Next, by \eqref{eq-V}, we need to find the solution $\tilde{\gamma}_{i},i=1,5$ and $\tilde{\sigma}_{j},j=2,3$ as a function of $\gamma_{5},\sigma_3,\sigma_2,\gamma_1,Z_{a}h,Z_{b}h$ and $V_a$ such that
\begin{align*}
&\tilde{\Phi}_5(\tilde{\gamma}_5,\mathcal{F}(\tilde{\sigma}_3,\tilde{\sigma}_2;\tilde{\Phi}_1(\tilde{\gamma}_1;V_a
+\widetilde{\mathcal{V}}(V_a,Z_{a}h)Z_{a}h)))\\
&=\tilde{\Phi}_5(\gamma_5,\mathcal{F}(\sigma_3,\sigma_2;\tilde{\Phi}_1(\gamma_1;V_a)))
+\widetilde{\mathcal{V}}(V_b,Z_{b}h)Z_{b}h.
\end{align*}

It follows from the implicit function theorem that $\tilde{\gamma}_{i},i=1,5$ and $\tilde{\sigma}_{j},j=2,3$ can be solved as a $C^2$-function of $(\gamma_{5},\sigma_3,\sigma_2,\gamma_1,Z_{a}h,Z_{b}h,V_a)$ uniquely. Moreover, we can obtain
\begin{align*}
\tilde{\gamma}_{i}(\gamma_{5},\sigma_3,\sigma_2,\gamma_1,Z_{a}h,Z_{b}h,V_a)
&=O(1)|Z_a-Z_b|h+\tilde{\gamma}_{i}(\gamma_{5},\sigma_3,\sigma_2,\gamma_1,Z_{a}h,Z_{a}h,V_a)\\
&=O(1)|Z_a-Z_b|h+O(1)Z_{a}h+\tilde{\gamma}_{i}(\gamma_{5},\sigma_3,\sigma_2,\gamma_1,0,0,V_a)\\
&=\gamma_i+O(1)Z_{a}h+O(1)|\gamma_4|h.
\end{align*}

The proof of $\eqref{estimate-Re3}_2$ can be derived in the same way. It completes the proof of this proposition.
\end{proof}

\section{Global entropy solutions of the steady exothermically reacting Euler equations}\label{sec:GE}
Thanks to the local estimates obtained in section \ref{sec:LE}, in this section, we will introduce the fractional-step Glimm scheme and a Glimm-type functional to construct the approximate solutions for the initial boundary value problem \eqref{eq-reaction-2} and \eqref{I}-\eqref{B},
by deriving global estimates on the non-reacting step and the reacting step. 
With these in hand, the global existence of entropy solutions with a strong contact discontinuity is obtained.

\subsection{The Glimm fractional-step scheme}\label{Construction}

As shown in Fig. \ref{Points}, we use the notations in \eqref{A}-\eqref{N}, and let $h>0$ and $s>0$ be the step-length in the $x$ and $y$ directions respectively.

The construction of the fractional-step scheme for the inhomogeneous system \eqref{eq-reaction-2} is as follows.

\begin{figure}[ht]
\centering
\setlength{\unitlength}{1mm}
\begin{picture}
(50,48)(0,22)
\linethickness{1pt}
\put(0,62){\line(0,-1){40}}
\put(0,62){\line(5,-2){32}}
\put(24,52.3){\line(0,-1){30}}
\put(48,22){\line(0,1){36.3}}
\put(24,52.3){\line(4,1){24}}
\put(36,55){\vector(-1,3){3}}
\put(3,60.9){\line(5,2){4}}
\put(7,59.2){\line(5,2){4}}
\put(11,57.6){\line(5,2){4}}
\put(15,56){\line(5,2){4}}
\put(19,54.4){\line(5,2){4}}
\put(11,57.6){\vector(1,2){4}}
\put(28,53.3){\line(1,2){2}}
\put(32,54.3){\line(1,2){2}}
\put(36,55.3){\line(1,2){2}}
\put(40,56.3){\line(1,2){2}}
\put(44,57.3){\line(1,2){2}}
\put(11,68){$\textbf{n}_{k-1}$}
\put(33,65){$\textbf{n}_{k}$}
\put(9,53){$\Gamma_{k-1}$}
\put(37,52){$\Gamma_{k}$}
\put(-1,65){$P_{k-1}$}
\put(23,54){$P_{k}$}
\put(49,60){$P_{k+1}$}
\put(29,51){$\omega_{k}$}
\put(8,34){$\Omega_{k-1,h}$}
\put(34,34){$\Omega_{k,h}$}
\end{picture}
\caption{The Glimm fractional-step scheme.}\label{Points}
\end{figure}
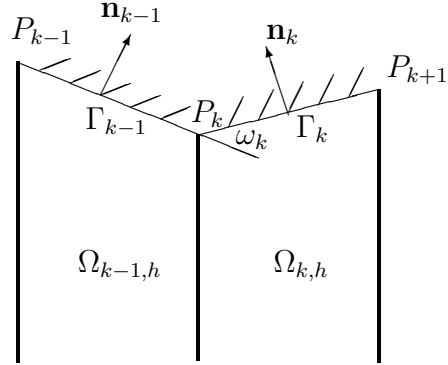

By \eqref{assum-g}, the boundary $y=g(x)$ is a perturbation of the straight wall. It means that for sufficiently small $\delta_0>0$, we have
\[
\sup_{x\ge 0}|g'(x)|<\delta_0.
\]
Therefore,
\begin{equation}\label{g}
m:=\sup_{k\geq 0}\Big\{\frac{|y_{k+1}-y_{k}|}{h}\Big\}<\delta_0.
\end{equation}

Let ${y}^{(0)}$ be given by \eqref{I}.
Choose $s$ such that ${y}^{(0)}/s=2N$ is an even number, and the following Courant--Friedrichs--Lewy condition holds:
\[
\frac{s}{h}>\max_{j=1,5}\Big(\sup_{U\in O_{\epsilon}(U_{1}^{(0)})\cup O_{\epsilon}(U_{2}^{(0)})}|\lambda_j(U)|\Big)+m.
\]

For any positive integer $k$ and negative integer $n$, \emph{i.e.}, $k\geq1$ and $n\leq-1$, define
\[
y_{k,n}=y_k+(2n+1+\theta_k)s,
\]
where $\theta_k$ is randomly chosen in $(-1,1)$.  Define 
\[
P_{k,n}=(kh,y_{k,n}),
\]
to be the mesh points.

Now we can define the approximate solutions $U_{h,\theta}$ in $\Omega_h$, where $\theta=(\theta_1,\theta_2,\cdots)$, inductively as follows.

First, for initial data $U_0(y)$ and for $y\in(2ns,(2n+2)s)$, let
\[
U_{h,0}(y)=\frac{1}{2s}\int_{2ns}^{(2n+2)s} U_0(y)dy.
\]

Second, assume that $U_{h,\theta}$ has been constructed in $\{0\leq x<kh\}\cap\Omega_h$,
then for $y\in(y_k+2ns,y_k+2(n+1)s)$, define $U_{k,n}^0$ and $\widetilde{U}_{k,n}^0$ such that
\begin{equation}\label{step-h}
\begin{cases}
&U_{k,n}^0:=U_h(kh-,y_{k,n}),\\
&W(\widetilde{U}_{k,n}^0):=W(U_{k,n}^0)+G(U_{k,n}^0)h.
\end{cases}
\end{equation}

Now we are going to define $U_{h,\theta}$ in $\Omega_{k,h}$.

The first case is the Riemann problem with the boundary. Let $T_{k,0}$ be the diamond with the vertices that $(kh,y_k),(kh,y_k-s),((k+1)h,y_{k+1}-s)$, and $((k+1)h,y_{k+1})$.
Then  $U_{h,\theta}=U_{k,0}$ in $T_{k,0}$ is the solution of the following Riemann problem:
\begin{equation}\label{T0}
\left\{
\begin{array}{ll}
W(U_{k,0})_{x}+H(U_{k,0})_y=0, & \text{in $T_{k,0}$},\\
U_{k,0}|_{x=kh}=\widetilde{U}_{k,-1}^0, &y_k-s<y<y_k,\\
(u_{k,0},v_{k,0})\cdot \textbf{n}_k=0,  &\text{on $\Gamma_k$}.
\end{array}
\right.
\end{equation}

The second case is the Riemann problem without the boundary. For $n\leq -1$, let $T_{k,n}$ be the diamond with the vertices that $(kh,y_k+(2n+1)s),(kh,y_k+(2n-1)s),((k+1)h,y_{k+1}+(2n-1)s)$, and $((k+1)h,y_{k+1}+(2n+1)s)$. Then $U_{h,\theta}=U_{k,n}$ in $T_{k,n}$ is the solution of the following Riemann problem:
\begin{equation}\label{Tn}
\left\{
\begin{array}{ll}
W(U_{k,n})_{x}+H(U_{k,n})_y=0 \qquad \text{in $T_{k,n}$},\\
U_{k,n}|_{x=kh}=
\begin{cases}
\widetilde{U}_{k,n}^0,\quad \quad y_k+2ns<y<y_k+(2n+1)s,\\
\widetilde{U}_{k,n-1}^0,\quad y_k+(2n-1)s<y<y_k+2ns.
\end{cases}
\end{array}
\right.
\end{equation}

Therefore, we constructed an approximate solution $U_{h,\theta}$ globally in $\Omega_h$ provided that we can obtain a uniform bound of $U_{h,\theta}$, which is the main objective in the remaining part of this section.

\subsection{Glimm-type functional}\label{subsec:3.2}
In order to obtain a uniform bound of $U_{h,\theta}$, let us introduce the Glimm-type functional in this subsection.
Assume that $U_{h,\theta}$ has been defined in $\{0\leq x<kh\}\cap\Omega_h$ and the following conditions are satisfied:
\begin{enumerate}[leftmargin=1]
\item[$A_{1}(k-1)$:]In each $\Omega_{h,i},\, 0\leq i\leq k-1$, there is a strong contact discontinuity $y=\chi^{(i)}$ with strength $(\sigma_{2}^{(i)},\sigma_{3}^{(i)},\gamma_{4}^{(i)})$ so that $\sigma_{j}^{(i)}\in O_{\hat{\epsilon}}(\sigma_{j0}), j=2,3,\, \gamma_{4}^{(i)}\in O_{\hat{\epsilon}}(0)$. $y=\chi^{(i)}$ divides $\Omega_{h,i}$ into two subregions: $\Omega_{h,i}^{(1)}$ and $\Omega_{h,i}^{(2)}$, where $\Omega_{h,i}^{(2)}$ is the region bounded by $y=\chi^{(i)}$ and $\Gamma_i$.
\item[$A_{2}(k-1)$:] $U_{h,\theta}|_{\Omega_{h,i}^{(1)}}\in O_{\epsilon}({U}_1^{(0)}),$\,\, and  $U_{h,\theta}|_{\Omega_{h,i}^{(2)}}\in O_{\epsilon}({U}_2^{(0)}),\,\, 0\leq i\leq k-1.$
\item[$A_{3}(k-1)$:] $\{\chi^{(i)}\}^{k-1}_{i=0}$ together forms $y=\chi_{h,\theta}(x)$, which is the strong contact discontinuity in $\{0\leq x<kh\}\cap\Omega_h$ and emanating from the point $(0,{y}^{(0)})$.
\end{enumerate}

Then we shall prove that $U_{h,\theta}$ defined in $\Omega_{h,k}$ by section \ref{Construction} satisfies $A_1(k),A_2(k)$ and $A_3(k)$. From the construction in section \ref{Construction}, there exists a strong contact discontinuity $y=\chi^{(k)}$ in a diamond $T_{k,n}$. We extend $\chi_{h,\theta}$ to $\Omega_{h,k}$ such that $\chi_{h,\theta}$=$\chi^{(k)}$ in $\Omega_{h,k}$ and define $\Omega_{h,k}^{(1)}$ and $\Omega_{h,k}^{(2)}$ in the same way as in $A_1(k-1)$. So it is sufficient to show that $A_2(k)$ holds such that
\[
U_{h,\theta}|_{\Omega_{h,k}^{(i)}}\in O_{\epsilon}({U}_i^{(0)}),\,\, i=1,2,\quad
\sigma_{j}^{(k)}\in O_{\hat{\epsilon}}(\sigma_{j0}),\,\, j=2,3,\quad \gamma_{4}^{(k)}\in O_{\hat{\epsilon}}(0).
\]

To achieve this, as in \cite{Glimm1965}, we introduce the mesh curves to establish the bound on the total variation of $U_{h,\theta}$.
\begin{definition}
A k-mesh curve J is a piecewise unbounded linear curve lying in the strip $\{(k-1)h\le x\le(k+1)h\}$ and consists of the diamond boundaries of the form $P_{k,n-1}N(\theta_{k+1},n)$, $P_{k,n-1}S(\theta_k,n)$, $S(\theta_k,n)P_{k,n}$, and $N(\theta_{k+1},n)P_{k,n}$,
where
\[
N(\theta_{k+1},n)=
\left\{
\begin{array}{ccc}
P_{k+1,n}&\mbox{$\theta_{k+1}\le 0$},\\
P_{k+1,n-1}&\mbox{$\theta_{k+1}> 0$},
\end{array}
\right.
\quad
S(\theta_{k},n)=
\left\{
\begin{array}{ccc}
P_{k-1,n-1}&\mbox{$\theta_{k}\le 0$},\\
P_{k+1,n}&\mbox{$\theta_{k}> 0$}.
\end{array}
\right.
\]
\end{definition}

\begin{definition}
We call mesh curve $I$ is an immediate successor to mesh curve $J$, if all but one mesh points of $I$ are on $J$ and $I$ lies on the right hand side of $J$.
\end{definition}

Then, we define the Glimm-type functional $F(J)$ on J.
\begin{definition}\label{def-f}
Let
\[
F(J)=L(J)+K Q(J),
\]
with
\begin{align*}
&L(J)=L_{c}(J)+L^1(J)+L^2(J),\\
&L_{c}(J)=C^{*}_{1}(|\sigma_2^J-\sigma_{20}|+|\sigma_3^J-\sigma_{30}|)+C^{*}_{2}|\gamma_4^J|,\\
&L^1(J)=K^{*}_{11}L^{1}_{1}(J)+K^{*}_{12}L^{1}_{2}(J)+K^{*}_{13}L^{1}_{3}(J)+K^{*}_{14}L^{1}_{4}(J)
+K^{*}_{15}L^{1}_{5}(J),\\
&L^2(J)=K^{*}_{20}L_{0}(J)+L^{2}_{1}(J)+K^{*}_{22}L^{2}_{2}(J)+K^{*}_{23}L^{2}_{3}(J)+K^{*}_{24}L^{2}_{4}(J)
+K^{*}_{25}L^{2}_{5}(J),\\
&Q(J)=\sum\{|\alpha_j||\beta_i|:\text{both weak waves $\alpha_j$ and $\beta_i$ across $J$ and approach, $i,j\neq 4$.}\},\\
&L_0(J)=\sum\{|\omega_k(P_k)|:P_k \in \Gamma_J\},\quad \Gamma_J=\{P_k=(kh,y_k)\, :\,P_k \in J^+\cap\partial\Omega_{h},k\ge 0\},\\
&L^{i}_j(J)=\sum\{|\alpha_j|:\text{$\alpha_j$ across $J$ in region $\Omega^{(i)}_{h,k-1}\cup\Omega^{(i)}_{h,k}$,\, $i=1,2,\, j=1,2,3,4,5.$}\},\
\end{align*}
where $\sigma_2^J,\sigma_3^J$ and $\gamma_4^J$ stand for the strength of the strong contact discontinuity across $J$, and $J^+$ denotes the subregion of $\Omega_{h}$ such that all the points in $J^+$ lie at the right hand side of $J$.
\end{definition}

The positive constants $C^{*}_{1},C^{*}_{2}$ and $K$ in Definition \ref{def-f} will be defined later. The other constants are given in the following lemma.
\begin{lemma}\label{lemma-choice}
There exist positive constants $K_{1i}^{*},i=1,2,3,4,5,$ and $K^{*}_{2i},i=0,2,3,4,5,$ such that
\begin{align*}
&K^{*}_{20}>|K_{b0}|,\quad K^{*}_{2i}>|K_{bi}|,i=2,3,5,\quad K^{*}_{24}>C^{*}_{2},\quad K^{*}_{14}>C_{2}^{*},\\
&K^{*}_{11}<\frac{1-K^{*}_{25}|K_{25}|}{|K_{21}|},\quad K^{*}_{15}>K^{*}_{25}|K_{15}|+K^{*}_{11}|K_{11}|,
\end{align*}
and $K_{1i}^{*},i=2,3,$ are arbitrarily large positive constants.
\end{lemma}
\begin{proof}
By Proposition \ref{prop-in-b} and Proposition \ref{prop-in-c1}, we know that  $K_{b5}=1$ and $|K_{25}|<1$.
Hence there exists a constant $K^*_{25}$ such that
   $K_{b5}<K^*_{25}<1/|K_{25}|$. Then we can choose a positive constant $K^*_{11}$ satisfying
  \[ 0<K^{*}_{11}<\frac{1-K^{*}_{25}|K_{25}|}{|K_{21}|}.\]
  This completes the proof.
\end{proof}

\subsection{Global estimates of the approximate solutions}
In this subsection, we will show that the functional $F(J)$ is decreasing to establish the global estimates of the approximate solutions. First let us consider the estimates on the non-reacting step.
\begin{proposition}\label{prop-f}
Suppose that $g(x)$ satisfies \eqref{g}, and suppose that I and J are two k-mesh curves such that J is an immediate successor of I. If
\[
U_h|_{I\cap(\Omega_{h,k-1}^{(i)}\cup\Omega_{h,k}^{(i)})}\in O_{\epsilon}({U}_i^{(0)}),\quad i=1,2; \quad |\sigma^{I}_{j}-\sigma_{j0}|<\hat{\epsilon},\quad j=2,3;\quad |\gamma_{4}^{I}|<\hat{\epsilon},
\]
for some $\epsilon,\hat{\epsilon}>0$, then there exists $\tilde{\epsilon}>0$ such that if $F(I)\leq\tilde{\epsilon}$, then it holds that
\begin{equation}\label{FJ}
F(J)\leq F(I).
\end{equation}
\end{proposition}
\begin{proof}
Let $\Lambda$ be the diamond that is formed by I and J. Then assume that $I=I_0\cup I'$ and $J=I_0\cup J'$ such that  $\partial\Lambda=I'\cup J'$. We will show this proposition case by case depending on the location of $\Lambda$.

\begin{figure}[ht]
\begin{center}
\setlength{\unitlength}{1mm}
\begin{picture}(50,64)(0,6)
\linethickness{1pt}
\put(0,4){\line(0,1){72}}
\put(24,4){\line(0,1){72}}
\put(48,4){\line(0,1){72}}
\put(0,52){\line(3,2){20}}
\put(0,52){\line(3,-1){20}}
\put(0,52){\line(5,2){20}}
\put(0,18){\line(1,1){20}}
\put(0,18){\line(5,3){20}}
\put(0,18){\line(4,-1){20}}
\put(24,40){\line(5,-2){22}}
\put(24,40){\line(3,2){22}}
\put(24,40){\line(5,1){22}}
\thinlines
\qbezier(0,30)(12,42)(24,54)
\qbezier(0,30)(12,27)(24,24)
\qbezier(24,24)(36,32)(48,40)
\qbezier(24,54)(36,47)(48,40)
\put(24,24){\line(1,-1){15}}
\put(24,54){\line(-1,2){10}}
\put(10,63){$\beta_5$}
\put(10,53){$\beta_{2(3,4)}$}
\put(10,44.5){$\beta_1$}
\put(10,34){$\alpha_5$}
\put(10,22){$\alpha_{2(3,4)}$}
\put(10,11){$\alpha_1$}
\put(40,54){$\gamma_5$}
\put(35,40){$\gamma_{2(3,4)}$}
\put(40,30){$\gamma_1$}
\put(2,35){$I'$}
\put(32,25){$J'$}
\put(32,10){$I_0$}
\put(16.5,70){$I_0$}
\end{picture}
\end{center}
\caption{Case 1: in the interior of $\Omega_h$.}\label{8}
\end{figure}
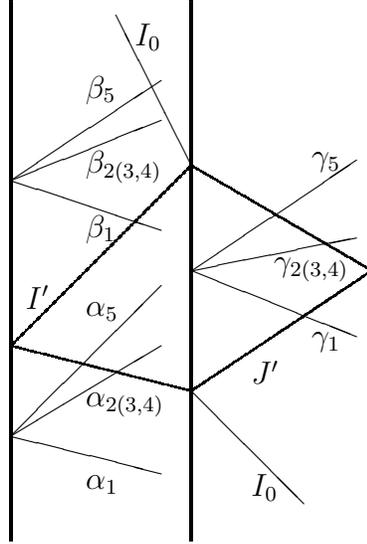

\smallskip
{\bf Case 1} (Fig. \ref{8}):
$\Lambda$ lies in the interior of $\Omega_h$ and only weak waves enter $\Lambda$. Without loss of the generality, we assume that $\Lambda$ lies in region (1). Denote $Q(\Lambda)=\Delta(\boldsymbol{\alpha}^*,\boldsymbol{\beta}^*)$, where $\Delta(\boldsymbol{\alpha}^*,\boldsymbol{\beta}^*)$ is defined in \eqref{estimate-in-w}. Then by Proposition \ref{prop-in-w}, we have
\begin{align*}
&L^1(J)-L^1(I)\leq M(K_{11}^{*}+K_{12}^{*}+K_{13}^{*}+K_{15}^{*})Q(\Lambda),\\
&Q(J)-Q(I)\leq (ML(I_0)-1)Q(\Lambda).
\end{align*}

Note that $F(I)\leq\tilde{\epsilon}$ for sufficiently small $\tilde{\epsilon}$, then it holds that
\begin{align*}
F(J)-F(I)&\leq \Big(M(K_{11}^{*}+K_{12}^{*}+K_{13}^{*}+K_{15}^{*})+K(ML(I_0)-1)\Big)Q(\Lambda)\\
&\leq -\frac{1}{2}Q(\Lambda),
\end{align*}
provided that K is suitably large.

\smallskip
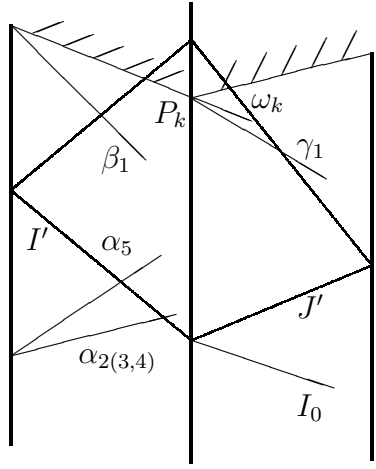
\begin{figure}[ht]
\begin{center}
\setlength{\unitlength}{1mm}
\begin{picture}
(50,60)(0,5)
\linethickness{1pt}
\put(0,62){\line(0,-1){56}}
\put(0,18){\line(4,1){22}}
\put(0,18){\line(3,2){20}}
\put(0,62){\line(1,-1){18}}
\put(0,62){\line(5,-2){32}}
\put(24,65){\line(0,-1){61.3}}
\put(48,4){\line(0,1){54.3}}
\put(24,52.3){\line(5,-3){18}}
\put(24,52.3){\line(4,1){24}}
\put(3,60.9){\line(5,2){4}}
\put(7,59.2){\line(5,2){4}}
\put(11,57.6){\line(5,2){4}}
\put(15,56){\line(5,2){4}}
\put(19,54.4){\line(5,2){4}}
\put(28,53.3){\line(1,2){2}}
\put(32,54.3){\line(1,2){2}}
\put(36,55.3){\line(1,2){2}}
\put(40,56.3){\line(1,2){2}}
\put(44,57.3){\line(1,2){2}}
\thinlines
\qbezier(0,40)(6,45)(24,60)
\qbezier(0,40)(6,35)(24,20)
\qbezier(24,60)(28,55)(48,30)
\qbezier(24,20)(36,25)(48,30)
\put(24,20){\line(3,-1){19}}
\put(2,32){$I'$}
\put(38,23){$J'$}
\put(38,10){$I_0$}
\put(19,49){$P_{k}$}
\put(32,51){$\omega_{k}$}
\put(38,45){$\gamma_1$}
\put(12,43){$\beta_1$}
\put(12,32){$\alpha_5$}
\put(9,17){$\alpha_{2(3,4)}$}
\end{picture}
\end{center}
\caption{Case 2: near the boundary.}\label{9}
\end{figure}

{\bf Case 2} (Fig. \ref{9}):
$\Lambda$ touches the approximate boundary $\partial\Omega_h$, and $\Gamma_I=\Gamma_J\cup\{P_k\}$ for certain $k$.
Using Proposition \ref{prop-in-b}, we can obtain
\begin{align*}
&L_0(J)-L_0(I)=-|\omega_k|,\\
&L_{1}^{2}(J)-L_{1}^{2}(I)\leq |K_{b0}||\omega_{k}|+\sum\limits_{i=2,3,5}|K_{bi}||\alpha_i|,\\
&L_{i}^{2}(J)-L_{i}^{2}(I)=-|\alpha_i|,\quad i=2,3,4,5.\\
&Q(J)-Q(I)\leq (|K_{b0}||\omega_{k}|+\sum\limits_{i=2,3,5}|K_{bi}||\alpha_i|)L(I_0).
\end{align*}
It implies that
\[
L^{2}(J)-L^{2}(I)\leq (|K_{b0}|-K^{*}_{20})|\omega_{k}|+\sum\limits_{i=2,3,5}(|K_{bi}|-K^{*}_{2i})|\alpha_i|.
\]

Therefore, if $F(I)\leq\tilde{\epsilon}$ for sufficiently small $\tilde{\epsilon}$, then it holds that $F(J)\leq F(I)$ by the choice of $K^{*}_{20}$ and $K^{*}_{2i}$ in Lemma \ref{lemma-choice}.

\smallskip
{\bf Case 3.1} (Fig. \ref{10}):
The diamond $\Lambda$ covers $\chi^{(k-1)}$ and the weak waves lying in region (2) interact with $\chi^{(k-1)}$ from the above. By applying Proposition \ref{prop-in-c1}, we have
\begin{align*}
&L_{1}^{1}(J)-L_{1}^{1}(I)\leq |K_{21}||\beta_1|+M\Delta'(\alpha_5,\boldsymbol{\beta}^{*}),\\
&L_{i}^{2}(J)-L_{i}^{2}(I)=-|\beta_i|,\quad i=1,2,3,4,\\
&L_{5}^{2}(J)-L_{5}^{2}(I)\leq |K_{25}||\beta_1|+M\Delta'(\alpha_5,\boldsymbol{\beta}^{*}),\\
&|\sigma_j^{J}-\sigma_j^{I}|\leq |K_{2i}||\beta_1|+|\beta_j|+M\Delta'(\alpha_5,\boldsymbol{\beta}^{*}),\quad j=2,3\\
&|\gamma_4^{J}-\gamma_4^{I}|=|\beta_4|,\\
&Q(J)-Q(I)\leq (|K_{21}|+|K_{25}|)|\beta_1|L(I_0)+(ML(I_0)-1)\Delta'(\alpha_5,\boldsymbol{\beta}^{*}).
\end{align*}
It implies that
\begin{align*}
L^1(J)+L^2(J)-L^1(I)-L^2(I)&\leq (K^{*}_{11}|K_{21}|+K^{*}_{25}|K_{25}|-1)|\beta_1|
-\sum\limits_{i=2,3,4}K^{*}_{2i}|\beta_i|
 +M\Delta'(\alpha_5,\boldsymbol{\beta}^{*}).
\end{align*}

Therefore, if $F(I)\leq\tilde{\epsilon}$ for sufficiently small $\tilde{\epsilon}$, then from the facts that $K^{*}_{11}|K_{21}|+K^{*}_{25}|K_{25}|<1$ and that $K_{24}^{*}>C_{2}^{*}$ by Lemma \ref{lemma-choice}, it holds that $F(J)\leq F(I)$ by choosing suitably small $C_{1}^{*}$ and suitably large $K$.

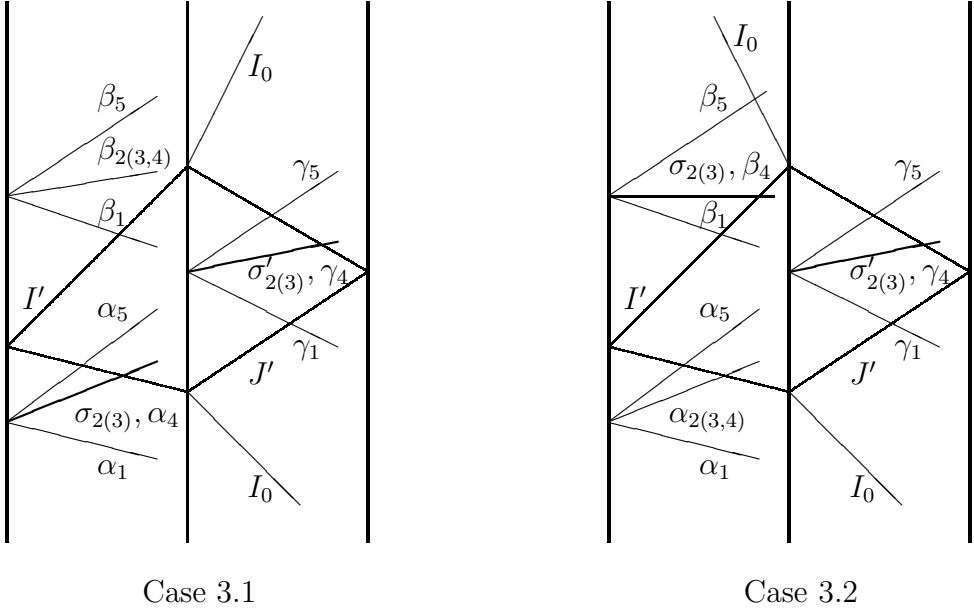
\begin{figure}[ht]
\begin{center}
\setlength{\unitlength}{1mm}
\begin{picture}
(70,72)(30,2)
\linethickness{1pt}
\put(0,4){\line(0,1){72}}
\put(24,4){\line(0,1){72}}
\put(48,4){\line(0,1){72}}
\put(0,50){\line(3,2){20}}
\put(0,50){\line(3,-1){20}}
\put(0,50){\line(6,1){20}}
\put(0,20){\line(4,3){20}}
\thicklines
\put(0,20){\line(5,2){20}}
\thinlines
\put(0,20){\line(4,-1){20}}
\put(24,40){\line(2,-1){20}}
\put(24,40){\line(3,2){20}}
\thicklines
\put(24,40){\line(5,1){20}}
\thinlines
\qbezier(0,30)(12,42)(24,54)
\qbezier(0,30)(12,27)(24,24)
\qbezier(24,24)(36,32)(48,40)
\qbezier(24,54)(36,47)(48,40)
\put(24,24){\line(1,-1){15}}
\put(24,54){\line(1,2){10}}
\put(2,35){$I'$}
\put(32,25){$J'$}
\put(32,10){$I_0$}
\put(32,66){$I_0$}
\put(12,62){$\beta_5$}
\put(12,55){$\beta_{2(3,4)}$}
\put(12,46.7){$\beta_1$}
\put(12,34){$\alpha_5$}
\put(9,20){$\sigma_{2(3)},\alpha_{4}$}
\put(12,13){$\alpha_1$}
\put(38,53){$\gamma_5$}
\put(32,39){$\sigma'_{2(3)},\gamma_{4}$}
\put(38,29){$\gamma_1$}

\linethickness{1pt}
\put(80,4){\line(0,1){72}}
\put(104,4){\line(0,1){72}}
\put(128,4){\line(0,1){72}}
\put(80,50){\line(3,2){21}}
\put(80,50){\line(3,-1){20}}
\thicklines
\put(80,50){\line(1,0){22}}
\thinlines
\put(80,20){\line(4,3){20}}
\put(80,20){\line(5,2){20}}
\put(80,20){\line(4,-1){20}}
\put(104,40){\line(2,-1){20}}
\put(104,40){\line(3,2){20}}
\thicklines
\put(104,40){\line(5,1){20}}
\thinlines
\qbezier(80,30)(92,42)(104,54)
\qbezier(80,30)(92,27)(104,24)
\qbezier(104,24)(116,32)(128,40)
\qbezier(104,54)(116,47)(128,40)
\put(104,24){\line(1,-1){15}}
\put(104,54){\line(-1,2){10}}
\put(82,35){$I'$}
\put(112,25){$J'$}
\put(112,10){$I_0$}
\put(96.5,70){$I_0$}
\put(92,62){$\beta_5$}
\put(88,53){$\sigma_{2(3)},\beta_{4}$}
\put(92,46.5){$\beta_1$}
\put(92,34){$\alpha_5$}
\put(88,20){$\alpha_{2(3,4)}$}
\put(92,13){$\alpha_1$}
\put(118,53){$\gamma_5$}
\put(112,39){$\sigma'_{2(3)},\gamma_{4}$}
\put(118,29){$\gamma_1$}
\put(18,-4){Case 3.1}
\put(98,-4){Case 3.2}
\end{picture}
\end{center}
\caption{Near the strong contact discontinuity.}\label{10}
\end{figure}

{\bf Case 3.2} (Fig. \ref{10}):
The diamond $\Lambda$ covers $\chi^{(k-1)}$ and the weak waves lying in region (1) interact with $\chi^{(k-1)}$ from the below. By Proposition \ref{prop-in-c2}, we can obtain
\begin{align*}
&L_{1}^{1}(J)-L_{1}^{1}(I)\leq |K_{11}||\alpha_5|+M\Delta''(\boldsymbol{\alpha}^{*},\beta_1),\\
&L_{i}^{1}(J)-L_{i}^{1}(I)=-|\alpha_i|,\quad i=2,3,4,5,\\
&L_{5}^{2}(J)-L_{5}^{2}(I)\leq |K_{15}||\alpha_5|+M\Delta''(\boldsymbol{\alpha}^{*},\beta_1),\\
&|\sigma_j^{J}-\sigma_j^{I}|\leq |K_{1i}||\alpha_5|+|\alpha_j|+M\Delta''(\boldsymbol{\alpha}^{*},\beta_1),\quad j=2,3\\
&|\gamma_4^{J}-\gamma_4^{I}|=|\alpha_4|,\\
&Q(J)-Q(I)\leq (|K_{11}|+|K_{15}|)|\alpha_5|L(I_0)+(ML(I_0)-1)\Delta''(\boldsymbol{\alpha}^{*},\beta_1).
\end{align*}
It implies that
\begin{align*}
L^1(J)+L^2(J)-L^1(I)-L^2(I)&\leq (K^{*}_{11}|K_{11}|+K^{*}_{25}|K_{15}|-K^{*}_{15})|\alpha_5|
-\sum\limits_{i=2,3,4}K^{*}_{1i}|\alpha_i|
+M\Delta''(\boldsymbol{\alpha}^{*},\beta_1).
\end{align*}

So if $F(I)\leq\tilde{\epsilon}$ for sufficiently small $\tilde{\epsilon}$, then from the facts that $K^{*}_{11}|K_{11}|+K^{*}_{25}|K_{15}|<K^{*}_{15}$ and $K_{14}^{*}>C_{2}^{*}$ by Lemma \ref{lemma-choice}, it holds that $F(J)\leq F(I)$ by choosing suitably small $C_{1}^{*}$ and suitably large $K$.
\end{proof}

In order to analyze the effect of the exothermic reaction on the functionals $L$ and $Q$, as in \cite{ChenWagner2003}, we introduce a new mesh curve $\tilde{J}$, which, as a curve, is the same as the mesh
curve J, but upon which the states $\tilde{U}$ are the values of the states $U$ on $J$ after a single reaction step along $J$. 

Let $J_k$ and $\tilde{J}_k$ be the $k$-mesh curve lying in $\{kh\leq x\leq (k+1)h\}$.
By Proposition \ref{prop-f}, we have
\begin{corollary}
Suppose that $g(x)$ satisfies \eqref{g}. Let $\epsilon,\hat{\epsilon},\tilde{\epsilon}$ be the constants given in Proposition $\mathrm{\ref{prop-f}}$ such that the induction hypotheses $A_{1}(k-1)$-$A_{3}(k-1)$ hold. If $F(\tilde{J}_{k-1})\leq\tilde{\epsilon}$, then it holds that
\begin{equation}\label{estimate-FJ_k-1}
F(J_{k})\leq F(\tilde{J}_{k-1}).
\end{equation}
\end{corollary}

\smallskip
Next, let us consider the estimates on the reacting step.
\begin{proposition}\label{prop:3.2}
There exists a positive constant $M$ such that
\begin{equation}\label{estimate-J_k}
\begin{split}
&L(\tilde{J_{k}})\leq L(J_{k})+M e^{-lkh}h\|Z_{0}\|_{\infty}\big(L(J_k)+1\big),\\
&Q(\tilde{J_{k}})\leq Q(J_{k})+M e^{-lkh}h\|Z_{0}\|_{\infty}\big(L(J_k)+1\big)^2.
\end{split}
\end{equation}
It implies that
\begin{equation}\label{estimate-FJ_k-2}
F(\tilde{J_{k}})\leq F(J_{k})+M e^{-lkh}h\|Z_{0}\|_{\infty}\big(F(J_k)+2\big)^2.
\end{equation}
Here $\|\cdot\|_{\infty}$ stands for $L^{\infty}$ norm.
\end{proposition}
\begin{proof}
By lemma \ref{lemma-Re} and by the induction method, we can easily obtain that
\begin{equation}\label{3.11x}
\|Z_{h,\theta}(kh+,\cdot)\|_{\infty}\leq e^{-lkh}\|Z_0\|_{\infty}.
\end{equation}

Then we will consider the change of $L$ on the reaction step, which is the first inequality of \eqref{estimate-J_k}. The analysis is divided into three cases depending on the location of $\Lambda$.
\begin{enumerate}[(1)]
\item  $\Lambda$ lies in the interior of $\Omega_h$ so that only weak waves $\gamma$ go out of $\Lambda$ through ${J}_k$. Without loss of the generality, we assume that $\Lambda$ lies in region in region (1).  With notations in Proposition \ref{prop-Re1}, we use \eqref{estimate-Re1} to deduce the following,
\begin{align*}
&L^{1}_{i}(\tilde{J}_k)\leq L^{1}_{i}({J}_k)+Me^{-lkh}h\|Z_0\|_{\infty}|\boldsymbol{\gamma}^{*}|+M|\gamma_4|h,\\
&L^{1}_{4}(\tilde{J}_k)\leq L^{1}_{4}({J}_k)+Me^{-lkh}h\|Z_0\|_{\infty}|\boldsymbol{\gamma}^{*}|-l|\gamma_4|h.
\end{align*}
Then it holds that $L(\tilde{J_{k}})\leq L(J_{k})+M e^{-lkh}h\|Z_{0}\|_{\infty}|\boldsymbol{\gamma}^{*}|$ by choosing suitably large $K^{*}_{14}$.
\item $\Lambda$ covers a part of $\partial\Omega_h$. It follows from \eqref{estimate-Re2} that $L(\tilde{J_{k}})\leq L(J_{k})+M e^{-lkh}h\|Z_{0}\|_{\infty}.$
\item $\Lambda$ covers the strong contact discontinuity $\chi^{(k)}$ so that $\chi^{(k)}$ with strength $(\sigma_2^{(k)},{\sigma}_3^{(k)},\gamma_4^{(k)})$ goes out of $\Lambda$ through ${J}_k$. By \eqref{estimate-Re3}, we can obtain
\begin{align*}
&L^{1}_{1}(\tilde{J}_k)\leq L^{1}_{1}({J}_k)+Me^{-lkh}h\|Z_0\|_{\infty}+M|\gamma_4^{(k)}|h,\\
&L^{2}_{5}(\tilde{J}_k)\leq L^{2}_{5}({J}_k)+Me^{-lkh}h\|Z_0\|_{\infty}+M|\gamma_4^{(k)}|h,\\
&|\tilde{\sigma}_j^{(k)}-\sigma_j^{(k)}|\leq Me^{-lkh}h\|Z_0\|_{\infty}+M|\gamma_4^{(k)}|h,\,\,\,j=2,3,\\
&|\tilde{\gamma}_4^{(k)}-\gamma_4^{(k)}|\leq Me^{-lkh}h\|Z_0\|_{\infty}-l|\gamma_4^{(k)}|h.
\end{align*}
Then it holds that $L(\tilde{J_{k}})\leq L(J_{k})+M e^{-lkh}h\|Z_{0}\|_{\infty}$ by choosing suitably large $C_{2}^{*}$.
\end{enumerate}

Thus, Combining these three cases, we proved the first inequality of $\eqref{estimate-J_k}$.

The second estimate in (\ref{estimate-J_k}) and the estimate in (\ref{estimate-FJ_k-2}) can be derived in the same way.
The proof is complete.
\end{proof}

Now in order to obtain the uniform bound of the total variation of $U_{h,\theta}$, we introduce the following functional:
\[
F_{c}(J_k)=F(J_k)+K_z\sum_{j=k+1}^{\infty}e^{-ljh}h\|Z_{0}\|_{\infty},
\]
where constant $K_z$ will be defined later.

\begin{lemma}\label{lem:3.2}
There exist positive constants $K_z$ and $\tilde{\epsilon}$, such that if $F_{c}(\tilde{J}_{k-1})\leq\tilde{\epsilon}$, then it holds that
\begin{equation}\label{estimate-Fc}
F_{c}(\tilde{J}_k)\leq F_c(\tilde{J}_{k-1}),
\end{equation}
and
\[
U_h|_{\Omega_{h,k}^{(i)}}\in O_{\epsilon}({U}_i^{(0)}),\,\,i=1,2,\quad |\sigma^{(k)}_{j}-\sigma_{j0}|<\hat{\epsilon},\,\,j=2,3,\quad |\gamma_{4}^{(k)}|<\hat{\epsilon}.
\]
\end{lemma}

\begin{proof}
From the estimates \eqref{estimate-FJ_k-1} and \eqref{estimate-FJ_k-2}, we have
\[
F(\tilde{J_{k}})\leq F(\tilde{J}_{k-1})+M e^{-lkh}h\|Z_{0}\|_{\infty}\big(F(J_k)+2\big)^2.
\]
It implies
\begin{align*}
F_{c}(\tilde{J}_k)-F_{c}(\tilde{J}_{k-1})&=F(\tilde{J}_k)-F(\tilde{J}_{k-1})-K_z e^{-lkh}h\|Z_{0}\|_{\infty}\\
&\leq (M\big(F(J_k)+2\big)^2-K_z) e^{-lkh}h\|Z_{0}\|_{\infty}.
\end{align*}

Note that $F(J_k)<\tilde{\epsilon}$. So we can choose suitably large $K_z$  such that $F_{c}(\tilde{J}_k)\leq F_c(\tilde{J}_{k-1})$ and $|\sigma^{(k)}_{j}-\sigma_{j0}|<\hat{\epsilon}$, $j=2,3$, and $|\gamma_{4}^{(k)}|<\hat{\epsilon}$.

Next, for any $k\geq 0$, define $U_{h,\theta}(kh+,-\infty)=\lim\limits_{y\to-\infty}U_{h,\theta}(kh+,y)$. Then by the fact that  $\lim\limits_{y\to-\infty}Z_1(y)=0$ and from the construction of the approximate solutions, we have that
\[
U_h(kh+,-\infty)=\lim\limits_{y\to-\infty}U_0(y).
\]
Then by Lemma \ref{lemma-R1} and \eqref{diff-estimate-f}, for sufficiently small $\tilde{\epsilon}$, it holds that $U_h|_{\Omega_{h,k}^{(i)}}\in O_{\epsilon}({U}_i^{(0)}),\,\,i=1,2$.
\end{proof}

Based on Proposition \ref{prop-f},  Proposition \ref{prop:3.2}, and Lemma \ref{lem:3.2}, we have the following theorems on the uniform B.V. bound of the approximate solution $U_{h,\theta}$.
\begin{theorem}\label{theorem-Uh}
Under assumptions $\mathrm{(H1)}$-$\mathrm{(H3)}$, there exist positive constants $\delta_0$ and $C$ such that, if \eqref{assum-g}-\eqref{assum-u-2} hold,
then for any ${\theta}\in \prod_{k=1}^{\infty}(-1,1)$ and h, the modified Glimm scheme defines global approximate solutions $U_{h,\theta}$ in $\Omega_h$, which satisfy $A_{1}(k)$--$A_{3}(k)$ given in section \ref{subsec:3.2} for $k\geq 0$. In addition,
\begin{equation}\label{tv-u_h}
T.V.\{U_{h,\theta}(kh-,\cdot): (-\infty,y_k]\}\leq C\delta_0,
\end{equation}
for any $k\geq 0$ and
\[
|\chi_{h,\theta}(x')-\chi_{h,\theta}(x'')|\leq C(|x'-x''|+h),
\]
for any $x',x''\geq 0$.
\end{theorem}

Based on Theorem \ref{theorem-Uh}, now we can show the global existence of entropy  solutions of \eqref{eq-reaction-1} as follows.

\begin{proof}[Proof of Theorem \ref{theorem-exist}]
The convergence of the approximate solutions to a global entropy solution can be carried out in the standard way as the one in \cite{ChenWagner2003,Glimm1965,Zhang1999} by using the structure of the approximate solutions. Therefore, we can establish the global existence of entropy solutions of \eqref{eq-reaction-1}, \emph{i.e.}, Theorem \ref{theorem-exist}.
\end{proof}

\section{Error estimate of the quasi-one-dimensional approximation}\label{sec:EE}
In this section, we shall study the quasi-one-dimensional approximation of two-dimensional steady supersonic exothermically reacting Euler flows between the Lipschitz wall $g(x)$ and strong contact discontinuity $\chi(x)$. To do that, we first solve the quasi-one-dimensional model, and then introduce several integral identities of the approximate solutions to show that the distance between the wall and the strong contact discontinuity has positive lower and upper bounds. Then we introduce the integral average of the approximate solutions with respect to $y$, and find the equations which the integral average satisfies as $h\rightarrow0$. Based on them, the difference between the integral average of the weak solution and the solution of the quasi-one-dimensional system can be estimated by analyzing the error terms.

\subsection{Quasi-one-dimensional model}\label{sec:A2}
In this subsection, we shall establish the global existence of solution to quasi-one-dimensional model \eqref{eq-QO1}.

First system \eqref{eq-QO1} with initial data $U_{A,0}=(\rho_{A,0},u_{A,0},p_{A,0},Z_{A,0})^\top$ can be written equivalently as
\begin{equation}\label{eq-QO2}
\begin{cases}
\rho u A(x) = \rho_{A,0} u_{A,0} A(0),\\
u+\frac{A(x)}{\rho_{A,0} u_{A,0} A(0)} p = u_{A,0}+\frac{A(0)}{\rho_{A,0} u_{A,0} A(0)}p_{A,0}+\frac{1}{\rho_{A,0} u_{A,0} A(0)}\int_0^x A'(\tau)pd\tau,\\
\frac{\gamma p}{(\gamma-1)\rho}+\frac{1}{2}u^2 =\frac{\gamma p_{A,0}}{(\gamma-1)\rho_{A,0}}+\frac{1}{2}u_{A,0}^2+\frac{q_0}{\rho_{A,0} u_{A,0} A(0)}\int_0^x A(\tau)\rho\phi(T)Zd\tau,\\
Z =Z_{A,0} -\frac{1}{\rho_{A,0} u_{A,0} A(0)}\int_0^x A(\tau)\rho\phi(T)Zd\tau.
\end{cases}
\end{equation}

Then we have the following lemma.

\begin{lemma}\label{lemma-Quasi-One-Ex}
There exist positive constants $\delta_0,C,C_*,$ and $C^*$, such that if $|U_{A,0}-\hat{U}_2^{(0)}|\leq \delta_0$, with $\hat{U}_2^{(0)}=({\rho}_2^{(0)},{u}_2^{(0)},{p}_2^{(0)},0)^\top$ and $\int_{0}^{\infty}|A'(\tau)|d\tau\leq \delta_0$, then the system \eqref{eq-QO2} admits a unique global solution $U_A(x)$ satisfying that
\begin{equation}\label{estimate-QO}
\max\limits_{x\geq 0}|U_{A}(x)-\hat{U}_2^{(0)}|\leq C\delta_0,\quad Z_{A,0}e^{-C^*x} \leq Z_A \leq Z_{A,0}e^{-C_*x}.
\end{equation}
\end{lemma}
\begin{proof}
We use the following the iteration scheme to establish a sequence of functions convergent to a solution.
Let
$$
(\rho_A^{(0)},u_A^{(0)},p_A^{(0)})=(\rho_{A,0},u_{A,0},p_{A,0}),
$$
and $Z_A^{(0)}$ is given by the last equation of \eqref{eq-QO3} for $n=0$. Precisely, we have
\[
Z_A^{(0)} =Z_{A,0}\exp(-\frac{1}{\rho_{A,0} u_{A,0} A(0)}\int_0^x A(\tau)\rho_{A,0}\phi(T_{A,0})d\tau).
\]

Then for any $n\geq 1$, the functions $U_{A}^{(n)}(x)=(\rho_{A}^{(n)},u_A^{(n)},p_A^{(n)},Z_A^{(n)})$ are determined inductively by
\begin{equation}\label{eq-QO3}
\left\{
\begin{array}{l}
\rho_A^{(n)} u_A^{(n)} A(x) = \rho_{A,0} u_{A,0} A(0),\\
u_A^{(n)}+\frac{A(x)}{\rho_{A,0} u_{A,0} A(0)}p_A^{(n)} =u_{A,0}+\frac{A(0)}{\rho_{A,0} u_{A,0} A(0)}p_{A,0}+\frac{1}{\rho_{A,0} u_{A,0} A(0)}\int_0^x A'(\tau)p_A^{(n-1)}d\tau,\\
\frac{\gamma p_A^{(n)}}{(\gamma-1)\rho_A^{(n)}}+\frac{1}{2}(u_A^{(n)})^2 =\frac{\gamma p_{A,0}}{(\gamma-1)\rho_{A,0}}+\frac{1}{2}u_{A,0}^2
+\frac{q_0}{\rho_{A,0} u_{A,0} A(0)}\int_0^x A(\tau)\rho_A^{(n-1)}\phi(T_A^{(n-1)})Z_A^{(n-1)}d\tau,\\
Z_A^{(n)} =Z_{A,0} -\frac{1}{\rho_{A,0} u_{A,0} A(0)}\int_0^x A(\tau)\rho_A^{(n)}\phi(T_A^{(n)})Z_A^{(n)}d\tau.
\end{array}
\right.
\end{equation}

First, let us prove inductively that for any $n\geq 0,\,U_{A}^{(n)}(x)$ are well defined and that there exist positive constants $\delta_0$ and $C$, such that the following inequality holds
\begin{equation}\label{estimate-QO-n}
\max\limits_{x\geq 0}|U_{A}^{(n)}(x)-\hat{U}_2^{(0)}|\leq C\delta_0.
\end{equation}

Obviously, it is true for $n=0$. Now assume that the estimate \eqref{estimate-QO-n} holds for $n=k-1,k\geq 1$, then we have
\[
C_*\leq \frac{1}{\rho_{A,0} u_{A,0} A(0)}A(x)\rho_A^{(k-1)}\phi(T_A^{(k-1)}) \leq C^*,
\]
for some constants $C_*$ and $C^*$. The last equation of \eqref{eq-QO3} yields that
\[
Z_A^{(k-1)} =Z_{A,0}\exp(-\frac{1}{\rho_{A,0} u_{A,0} A(0)}\int_0^x A(\tau)\rho_A^{(k-1)}\phi(T_A^{(k-1)})d\tau).
\]
It implies that
\begin{equation}\label{estimate-Z-n-1}
Z_{A,0}e^{-C^*x} \leq Z_A^{(k-1)} \leq Z_{A,0}e^{-C_*x}.
\end{equation}

Let $\mathcal{H}(V_A^{(k)},A(x))=(\rho_A^{(k)} u_A^{(k)} A(x),\,u_A^{(k)}+\frac{A(x)}{\rho_{A,0} u_{A,0} A(0)} p_A^{(k)},\,\frac{\gamma p_A^{(k)}}{(\gamma-1)\rho_A^{(k)}}+\frac{1}{2}(u_A^{(k)})^2)^\top$ and $V_A^{(k)}=(\rho_{A}^{(k)},\,u_{A}^{(k)},\,p_{A}^{(k)})^\top$, then the first three equations of \eqref{eq-QO3} can be written as
\begin{equation}\label{eq-QO3-3}
\mathcal{H}(V_A^{(k)},A(x))=\mathcal{H}(V_{A,0},A(x))+\mathcal{H}_e(V_A^{(k-1)},V_{A,0},A(x),A(0),Z_A^{(k-1)}).
\end{equation}
where $V_{A,0}=(\rho_{A,0},\,u_{A,0},\,p_{A,0})^\top$, and the term $\mathcal{H}_e$ can be defined without confusion.

From the fact that $\int_{0}^{\infty}|A'(\tau)|d\tau\leq \delta_0$, and the estimate \eqref{estimate-Z-n-1}, we have
\begin{align*}
|A(0)-A(x)|&\leq \delta_0,\\
\Big|\int_0^x A'(\tau)p_A^{(k-1)}d\tau\Big|
&\leq \int_0^x |A'(\tau)||p_A^{(k-1)}-{p}_{2}^{(0)}|d\tau+\int_0^x |A'(\tau)|{p}^{(0)}_{2}d\tau\\
&\leq C\delta_0^2+ {p}^{(0)}_{2}\delta_0,
\end{align*}
and
\begin{align*}
\Big|\int_0^x A(\tau)\rho_A^{(k-1)}\phi(T_A^{(k-1)})Z_A^{(k-1)}d\tau\Big|
&\leq \int_0^x |A(\tau)|\rho_A^{(k-1)}\phi(T_A^{(k-1)})Z_A^{(k-1)}d\tau\\
&\leq C^*\rho_{A,0} u_{A,0} A(0) Z_{A,0}\int_0^x e^{-C_*\tau}d\tau\\
&\leq C^*\rho_{A,0} u_{A,0} A(0)\delta_0/C_*.
\end{align*}
Therefore, $\mathcal{H}_e$ is bounded by $O(1)\delta_0$. Then it follows from the implicit function theorem that $\max\limits_{x\geq 0}|V_{A}^{(k)}(x)-V_{A,0}|\leq C'\delta_0,$ by choosing suitably large $C'$ and suitably small $\delta_0$.

Again, from the last equation of \eqref{eq-QO3}, we have
\begin{equation}\label{estimate-Z-n}
Z_A^{(k)} =Z_{A,0}\mathrm{exp}(-\frac{1}{\rho_{A,0} u_{A,0} A(0)}\int_0^x A(\tau)\rho_A^{(k)}\phi(T_A^{(k)})d\tau),
\end{equation}
which implies $Z_{A,0}e^{-C^*x}\leq Z_A^{(k)}\leq Z_{A,0}e^{-C_*x}.$ So we obtain the estimate \eqref{estimate-QO-n} for $n=k$.

Second, we will show the convergence of the sequence $\{U_{A}^{(n)}(x)\}_{n=0}^{\infty}$.

Define
\[
w^{(n)}=-\frac{1}{\rho_{A,0} u_{A,0} A(0)}\int_0^x A(\tau)\rho_A^{(n)}\phi(T_A^{(n)})d\tau.
\]
Then by \eqref{estimate-Z-n}, we can obtain
\begin{align}\label{conver-Z}
|Z_A^{(n)}-Z_A^{(n-1)}|&=|Z_{A,0}(w^{(n)}-w^{(n-1)})\int_{0}^{1}\mathrm{exp}(sw^{(n)}+(1-s)w^{(n-1)})ds|\nonumber\\
&\leq\frac{Z_{A,0}e^{-C_*x}}{\rho_{A,0} u_{A,0} A(0)}\int_0^x A(\tau)|\rho_A^{(n)}\phi(T_A^{(n)})-\rho_A^{(n-1)}\phi(T_A^{(n-1)})|d\tau\nonumber\\
&\leq O(1)\delta_0\max\limits_{0\leq\tau\leq x}(|\rho_A^{(n)}-\rho_A^{(n-1)}|+|T_A^{(n)}-T_A^{(n-1)}|).
\end{align}

Next, by \eqref{eq-QO3-3}, it holds that
\begin{align}\label{eq-QO3-3n}
\mathcal{H}(V_A^{(n)},A(x))-\mathcal{H}(V_A^{(n-1)},A(x))=&
\mathcal{H}_e(V_A^{(n-1)},V_{A,0},A(x),A(0),Z_A^{(n-1)})\nonumber\\
&-\mathcal{H}_e(V_A^{(n-2)},V_{A,0},A(x),A(0),Z_A^{(n-2)}).
\end{align}
Noticing the fact that $\int_{0}^{\infty}|A'(\tau)|d\tau\leq \delta_0$, and the estimate \eqref{conver-Z}, we have
\begin{align*}
\Big|\frac{1}{\rho_{A,0} u_{A,0} A(0)}\int_0^x A'(\tau)(p_A^{(n-1)}-p_A^{(n-2)})d\tau\Big|\leq O(1)\delta_0\max\limits_{0\leq\tau\leq x}|p_A^{(n-1)}-p_A^{(n-2)}|,
\end{align*}
and
\begin{align*}
&\Big|\frac{q_0}{\rho_{A,0} u_{A,0} A(0)}\int_0^x A(\tau)(\rho_A^{(n-1)}\phi(T_A^{(n-1)})Z_A^{(n-1)}-\rho_A^{(n-2)} \phi(T_A^{(n-2)})Z_A^{(n-2)})d\tau\Big|\\
&\leq \frac{q_0}{\rho_{A,0} u_{A,0} A(0)}\int_0^x A(\tau)|\rho_A^{(n-1)}\phi(T_A^{(n-1)})-\rho_A^{(n-2)}\phi(T_A^{(n-2)})|Z_A^{(n-1)}d\tau\\
&\hspace{1em}+\frac{q_0}{\rho_{A,0} u_{A,0} A(0)}\int_0^x A(\tau)\rho_A^{(n-2)}\phi(T_A^{(n-2)})|Z_A^{(n-1)}-Z_A^{(n-2)}|d\tau\\
&\leq O(1)Z_{A,0}\int_0^xe^{-C_*\tau}d\tau\max\limits_{0\leq\tau\leq x}(|\rho_A^{(n-1)}-\rho_A^{(n-2)}|+|T_A^{(n-1)}-T_A^{(n-2)}|)\\
&\hspace{1em}+O(1)Z_{A,0}\int_0^x e^{-C_*\tau}\tau d\tau \max\limits_{0\leq\tau\leq x}(|\rho_A^{(n-1)}-\rho_A^{(n-2)}|+|T_A^{(n-1)}-T_A^{(n-2)}|)\\
&\leq O(1)\delta_0\max\limits_{0\leq\tau\leq x}(|\rho_A^{(n-1)}-\rho_A^{(n-2)}|+|T_A^{(n-1)}-T_A^{(n-2)}|).
\end{align*}

Therefore, the right-hand side of \eqref{eq-QO3-3n} is bounded by $O(1)\delta_0\max\limits_{0\leq\tau\leq x}|V_A^{(n-1)}-V_A^{(n-2)}|$. Then it follows from the implicit function theorem that
\begin{equation}\label{conver-V}
\max\limits_{x\geq 0}|V_{A}^{(n)}(x)-V_{A}^{(n-1)}(x)|\leq \frac{1}{2}\max\limits_{x\geq 0}|V_{A}^{(n-1)}(x)-V_{A}^{(n-2)}(x)|.
\end{equation}
by choosing suitably small $\delta_0$. Combining \eqref{conver-Z} and \eqref{conver-V}, we know that the limit $U_A(x)$ is an unique solution of \eqref{eq-QO2}, which belongs to $C([0,\infty),\mathbb{R}^4)$ and satisfies
\[
\max\limits_{x\geq 0}|U_{A}(x)-\hat{U}_2^{(0)}|\leq C\delta_0,\quad Z_{A,0}e^{-C^*x} \leq Z_A \leq Z_{A,0}e^{-C_*x}.
\]
\end{proof}

\subsection{Integral identities of the approximate solutions}
Let $U_{h,\theta}$ be the solution obtained by Theorem \ref{theorem-Uh}.
Let $\Omega_{i,h}$ be the domain with the boundaries that $x=(i-1)h$, $x=ih$, $y=g_{i-1,h}(x)$, and $y=\chi^{(i-1)}(x)$.
Let $b_{i-1}$ be the slope of $y=g_{i-1,h}(x)$. And let $s^{(i-1)}$ be the slope of $y=\chi^{(i-1)}(x)$ emanating from point $((i-1)h, y_{i-1,s})$, where $y_{i-1,s}=y_{i-1}+2n_{i-1}s$ with a negative integer $n_{i-1}$.
By applying the divergence theorem in domain $\Omega_{i,h}$ and using the Rankine-Hugoniot conditions, we have the following integral identities.
\begin{align}
&\int\limits_{y_{i-1,s}+s^{(i-1)}h}^{y_i}(\rho_{h,\theta} u_{h,\theta}) (ih-,y)dy-\int\limits_{y_{i-1,s}}^{y_{i-1}}(\rho_{h,\theta} u_{h,\theta}) ((i-1)h+,y)dy=0,\label{4.1s}\\
&\int\limits_{y_{i-1,s}+s^{(i-1)}h}^{y_i}(\rho_{h,\theta} u_{h,\theta}^2+p_{h,\theta})(ih-,y)dy-\int\limits_{y_{i-1,s}}^{y_{i-1}}(\rho_{h,\theta} u_{h,\theta}^2+p_{h,\theta}) ((i-1)h+,y)dy\nonumber\\
&\hspace{3em}+\int\limits_{(i-1)h}^{ih}\big(-b_{i-1}p_{h,\theta}(\tau,y)|_{y=g_{i-1,h}(\tau)}
+s^{(i-1)}p_{h,\theta}(\tau,y)|_{y=\chi^{(i-1)}(\tau)}\big)d\tau=0,\\
&\int\limits_{y_{i-1,s}+s^{(i-1)}h}^{y_i}\big(\rho_{h,\theta} u_{h,\theta}(\frac{\gamma p_{h,\theta}}{(\gamma-1)\rho_{h,\theta}}+\frac{1}{2}u_{h,\theta}^2+\frac{1}{2}v_{h,\theta}^2)\big)(ih-,y)dy
\nonumber \\
&\hspace{3em}-\int\limits_{y_{i-1,s}}^{y_{i-1}}\big(\rho_{h,\theta} u_{h,\theta}(\frac{\gamma p_{h,\theta}}{(\gamma-1)\rho_{h,\theta}}+\frac{1}{2}u_{h,\theta}^2+\frac{1}{2}v_{h,\theta}^2)\big)((i-1)h+,y)dy=0,\\
&\int\limits_{y_{i-1,s}+s^{(i-1)}h}^{y_i}(\rho_{h,\theta} u_{h,\theta} Z_{h,\theta}) (ih-,y)dy - \int\limits_{y_{i-1,s}}^{y_{i-1}}(\rho_{h,\theta} u_{h,\theta} Z_{h,\theta})((i-1)h+,y)dy=0.\label{4.4s}
\end{align}

Therefore, for any $x\in ((k-1)h,kh)$, summing over \eqref{4.1s}--\eqref{4.4s} with respect to $1\leq i\leq k-1$ respectively, we have that
\begin{align}
&\int_{\chi^{(k-1)}(x)}^{g_{k-1,h}(x)}(\rho_{h,\theta} u_{h,\theta})(x-,y)dy+\sum_{i=1}^{k-1}E_{1,i}(h,\theta)=\int_{{y}^{(0)}}^{0}(\rho_{h,\theta} u_{h,\theta})(0+,y)dy,\label{E1i}\\
&\int_{\chi^{(k-1)}(x)}^{g_{k-1,h}(x)}(\rho_{h,\theta} u_{h,\theta}^2+p_{h,\theta})(x-,y)dy+\sum_{i=1}^{k-1}E_{2,i}(h,\theta)
=\int_{{y}^{(0)}}^{0}(\rho_{h,\theta} u_{h,\theta}^2+p_{h,\theta})(0+,y)dy \nonumber \\
&\hspace{2em}+\sum_{i=1}^{k-1}\int_{(i-1)h}^{ih}\big(b_{i-1}p_{h,\theta}(\tau,y)|_{y=g_{i-1}(\tau)}-s^{(i-1)}
p_{h,\theta}(\tau,y)|_{y=\chi^{(i-1)}(\tau)}\big)d\tau \nonumber \\
&\hspace{2em}+\int_{(k-1)h}^{x}\big(b_{k-1}p_{h,\theta}(\tau,y)|_{y=g_{k-1}(\tau)}-s^{(k-1)}
p_{h,\theta}(\tau,y)|_{y=\chi^{(k-1)}(\tau)}\big)d\tau,\label{E2i}\\
&\int_{\chi^{(k-1)}(x)}^{g_{k-1,h}(x)}\big(\rho_{h,\theta} u_{h,\theta}(\frac{\gamma p_{h,\theta}}{(\gamma-1)\rho_{h,\theta}}+\frac{1}{2}u_{h,\theta}^2+\frac{1}{2}v_{h,\theta}^2)\big)(x-,y)dy
+\sum_{i=1}^{k-1}E_{3,i}(h,\theta)\nonumber \\
&\hspace{2em}=\int_{{y}^{(0)}}^{0}\big(\rho_{h,\theta} u_{h,\theta}(\frac{\gamma p_{h,\theta}}{(\gamma-1)\rho_{h,\theta}}+\frac{1}{2}u_{h,\theta}^2+\frac{1}{2}v_{h,\theta}^2)\big)(0+,y)dy,\label{E3i}\\
&\int_{\chi^{(k-1)}(x)}^{g_{k-1,h}(x)}(\rho_{h,\theta} u_{h,\theta} Z_{h,\theta})(x-,y)dy+\sum_{i=1}^{k-1}E_{4,i}(h,\theta)=\int_{{y}^{(0)}}^{0}(\rho_{h,\theta} u_{h,\theta} Z_{h,\theta})(0+,y)dy,\label{E4i}
\end{align}
where $E_{l,i}(h,\theta)(l=1,2,3,4)$ is the $l$th-component of the error term $E_{i}(h,\theta)$, and
\begin{align}
E_{i}(h,\theta)=\int\limits_{y_{i-1,s}+s^{(i-1)}h}^{y_{i}}(W(U_{h,\theta}))(ih-,y)dy
-\int\limits_{y_{i,s}}^{y_{i}}(W(U_{h,\theta}))(ih+,y)dy.
\end{align}

\smallskip
Now we are going to analyze the error terms $E_{i}(h,\theta)$ across the line $x=kh$.
Note the fact that
\begin{equation}\label{eq-ih-G}
W(U_{h,\theta}(ih+,y))=W(U_{h,\theta}(ih-,y_{i,n}))+G(U_{h,\theta}(ih-,y_{i,n}))h.
\end{equation}
in the interval $y_i+2ns< y <y_i+2(n+1)s, n\leq -1$, and the fact that
\[
\sum_{i=1}^{k-1}\int\limits_{y_{i,s}}^{y_{i}}G(U_{h,\theta}(ih-,y_{i,n}))hdy\to \int\limits_{0}^{x}\int\limits_{\chi(\tau)}^{g(\tau)}G(U(\tau,y))dyd\tau,
\]
when $h\to 0$ by the convergence of the approximate solutions.
Let
\[
\tilde{E}_{i}(h,\theta)=\int\limits_{y_{i-1,s}+s^{(i-1)}h}^{y_{i}}W(U_{h,\theta}(ih-,y))dy
-\int\limits_{y_{i,s}}^{y_{i}}W(U_{h,\theta_i}(ih-,y))dy,
\]
where $W(U_{h,\theta_i}(ih-,y))=W(U_{h,\theta}(ih-,y_{i,n}))$ in the interval $y_i+2ns< y <y_i+2(n+1)s$, for $n\leq -1$. Obviously
\begin{equation}
E_{i}(h,\theta)=\tilde{E}_{i}(h,\theta)-\int\limits_{y_{i,s}}^{y_{i}}G(U_{h,\theta}(ih-,y_{i,n}))hdy.
\end{equation}
Therefore, in order to estimate $E_{i}(h,\theta)$, we only need to estimate $\tilde{E}_{i}(h,\theta)$.

To get the more specific expression of $\tilde{E}_{i}(h,\theta)$, let
\[
d_{i}=\frac{s^{(i-1)}h-(y_i-y_{i-1})}{s}.
\]
Obviously, $d_i\in (-1,1)$, and is independent of $\theta_i$.

Now we will divide our analysis into two cases based on $d_i$. 

The first case is that $d_i<0$. In this case, if $\theta_i\in(-1,d_i+1)$, then we have $y_{i,s}=y_i+2n_{i-1}s$, and
\begin{align}\label{E-i-1}
\tilde{E}_{i}(h,\theta)=&\int\limits_{y_{i}+2n_{i-1}s}^{y_{i}}\big(W(U_{h,\theta}(ih-,y))
-W(U_{h,\theta_i}(ih-,y))\big)dy
+\int\limits_{y_{i}+(2n_{i-1}+d_i)s}^{y_{i}+2n_{i-1}s}W(U_{h,\theta}(ih-,y))dy.
\end{align}
If $\theta_i\in(d_i+1,1)$, then we have $y_{i,s}=y_i+2(n_{i-1}-1)s$, and
\begin{align}\label{E-i-2}
\tilde{E}_{i}(h,\theta)=&\int\limits_{y_{i}+2n_{i-1}s}^{y_{i}}\big(W(U_{h,\theta}(ih-,y))
-W(U_{h,\theta_i}(ih-,y))\big)dy\nonumber \\
&+\int\limits_{y_{i}+(2n_{i-1}+d_i)s}^{y_{i}+2n_{i-1}s}W(U_{h,\theta}(ih-,y))dy
-\int\limits_{y_{i}+2(n_{i-1}-1)s}^{y_{i}+2n_{i-1}s}W(U_{h,\theta_i}(ih-,y))dy.
\end{align}

The second case is that $d_i>0$. In this case, if $\theta_i\in(-1,d_i-1)$, then we have $y_{i,s}=y_i+2(n_{i-1}+1)s$, and
\begin{align}\label{E-i-3}
\tilde{E}_{i}(h,\theta)=&\int\limits_{y_{i}+2n_{i-1}s}^{y_{i}}\big(W(U_{h,\theta}(ih-,y))
-W(U_{h,\theta_i}(ih-,y))\big)dy\nonumber\\
&-\int\limits_{y_{i}+2n_{i-1}s}^{y_{i}+(2n_{i-1}+d_i)s}W(U_{h,\theta}(ih-,y))dy
+\int\limits_{y_{i}+2n_{i-1}s}^{y_{i}+2(n_{i-1}+1)s}W(U_{h,\theta_i}(ih+,y))dy.
\end{align}
If $\theta_i\in(d_i-1,1)$, then we have $y_{i,s}=y_i+2n_{i-1}s$, and
\begin{align}\label{E-i-4}
\tilde{E}_{i}(h,\theta)=&\int\limits_{y_{i}+2n_{i-1}s}^{y_{i}}\big(W(U_{h,\theta}(ih-,y))
-W(U_{h,\theta_i}(ih-,y))\big)dy
-\int\limits_{y_{i}+2n_{i-1}s}^{y_{i}+(2n_{i-1}+d_i)s}W(U_{h,\theta}(ih-,y))dy.
\end{align}

Let $\mathbf{1}_B$ be the characteristic function of set B.
Then, combining \eqref{E-i-1}-\eqref{E-i-4} together, we have
\begin{align}\label{eq-Ei}
&\tilde{E}_{i}(h,\theta)\nonumber \\
=&\mathbf{1}_{(-1,0)}(d_i)\Big(\int\limits_{y_{i}+2n_{i-1}s}^{y_{i}}\big(W(U_{h,\theta}(ih-,y))
-W(U_{h,\theta_i}(ih-,y))\big)dy\nonumber \\
&\qquad+\int\limits_{y_{i}+(2n_{i-1}+d_i)s}^{y_{i}+2n_{i-1}s}W(U_{h,\theta}(ih-,y))dy
-\mathbf{1}_{(d_i+1,1)}(\theta_i)\int\limits_{y_{i}+2(n_{i-1}-1)s}^{y_{i}+2n_{i-1}s}W(U_{h,\theta_i}(ih-,y))dy\Big)\nonumber\\
&+\mathbf{1}_{(0,1)}(d_i)\Big(\int\limits_{y_{i}+2n_{i-1}s}^{y_{i}}\big(W(U_{h,\theta}(ih-,y))
-W(U_{h,\theta_i}(ih-,y))\big)dy\\
&\qquad-\int\limits_{y_{i}+2n_{i-1}s}^{y_{i}+(2n_{i-1}+d_i)s}W(U_{h,\theta}(ih-,y))dy
+\mathbf{1}_{(-1,d_i-1)}(\theta_i)\int\limits_{y_{i}+2n_{i-1}s}^{y_{i}+2(n_{i-1}+1)s}W(U_{h,\theta_i}(ih-,y))dy\Big)\nonumber.
\end{align}

For the error term $\tilde{E}_{i}(h,\theta)$, we have the following lemma.
\begin{lemma}\label{lemma-tildeE}
For any $x\geq 0$, there exist a null set $\mathscr{N}_1\subset \prod_{k=1}^{\infty}(-1,1)$ and a subsequence $\{h_j\}_{j=1}^{\infty}$, 
which tends to $0$, such that when $h_j\to 0$, it holds that
\begin{equation}\label{estimate-E-i}
\sum_{i=1}^{k-1}\tilde{E}_{i}(h_j,\theta)\to 0,
\end{equation}
for any ${\theta}\in \prod_{k=1}^{\infty}(-1,1)\backslash \mathscr{N}_1$.
\end{lemma}
\begin{proof}
Note that $n_{i-1}$ is independent of $\theta_i$, then we have
\begin{align*}
&\frac{1}{2}\int\limits_{-1}^{1}\int\limits_{y_{i}+2n_{i-1}s}^{y_{i}}\big(W(U_{h,\theta}(ih-,y))
-W(U_{h,\theta_i}(ih-,y))dyd\theta_i\\
&=\frac{1}{2}\int\limits_{-1}^{1}\sum_{n=n_{i-1}}^{-1}\int\limits_{y_{i}+2ns}^{y_{i}+2(n+1)s}
\Big(W(U_{h,\theta}(ih-,y))-W(U_{h,\theta}(ih-,y_i+(2n+1+\theta_i))s)\Big)dyd\theta_i\\
&=\sum_{n=n_{i-1}}^{-1}\Big(\int\limits_{y_{i}+2ns}^{y_{i}+2(n+1)s}W(U_{h,\theta}(ih-,y))dy
-s\int\limits_{-1}^{1}W(U_{h,\theta}(ih-,y_i+(2n+1+\theta_i)s))d\theta_i\Big)\\
&=0.
\end{align*}

Next, note that if $d_i<0$, then $W(U_{h,\theta}(ih-,y))$ is a constant state independent of $\theta_i$ in the interval $(y_{i}+(2n_{i-1}+d_i)s,y_{i}+2n_{i-1}s)$; while if $d_i>0$, then $W(U_{h,\theta}(ih-,y))$ is a constant state independent of $\theta_i$ in the interval  $(y_{i}+2n_{i-1}s,y_{i}+(2n_{i-1}+d_i)s)$. Hence it follows from \eqref{eq-Ei} that
\[
\frac{1}{2}\int_{-1}^{1}\tilde{E}_{i}(h,\theta)d\theta_i=0.
\]

Therefore, we have
\[
\int|\sum_{i=1}^{k-1}\tilde{E}_{i}(h,\theta)|^2d\theta=\sum_{i=1}^{[x/h]}\int|\tilde{E}_{i}(h,\theta)|^2d\theta\leq Cx(\frac{s}{h})^2h.
\]
for some constant $C>0$. Then, we can show \eqref{estimate-E-i} by choosing a subsequence $\{h_j\}_{j=1}^{\infty}$ 
with $\sum_{j=1}^{\infty}h_j<\infty$.
\end{proof}

Moreover, by \eqref{E1i} and \eqref{estimate-E-i}, we also have the following lemma.
\begin{lemma}\label{lemma-g-chi}
There exist positive constants $A_1$ and $A_2$, such that for any $x\geq 0$,
$$
A_1\leq g(x)-\chi(x)\leq A_2.
$$
\end{lemma}

\subsection{Integral average of the approximate solutions}\label{subsec:4.3s}
If $\tau\in ((i-1)h,ih)$, we define the integral average of the approximate solutions as
\[\bar{U}_h(\tau-):=\frac{1}{g_{i-1,h}(\tau)-\chi^{(i-1)}(\tau)}\int_{\chi^{(i-1)}(\tau)}^{g_{i-1,h}(\tau)}U_{h,\theta}(\tau-,y)dy,\]
and the integral average of the approximate initial data as
\[\bar{U}_{h,0}:=\frac{1}{|{y}^{(0)}|}\int_{{y}^{(0)}}^{0}U_{h,0}(y)dy.\]

Now, we will derive the equation satisfied by the integral average of the weak solution.
Replacing the approximate solutions in equations \eqref{E1i}-\eqref{E4i} by the integral average of the approximate solutions, \eqref{E1i}-\eqref{E4i} can be rewritten as
\begin{align}
&(g_{k-1,h}(x)-\chi^{(k-1)}(x))\bar{\rho}_h \bar{u}_h+\sum_{i=1}^{k-1}E_{1,i}(h,\theta)\nonumber\\
=&-{y}^{(0)}\bar{\rho}_{h,0}\bar{u}_{h,0}
-\int\limits_{\chi^{(k-1)}(x)}^{g_{k-1,h}(x)}(\rho_{h,\theta}-\bar{\rho}_h)(u_{h,\theta}-\bar{u}_h)dy
+\int\limits_{{y}^{(0)}}^{0}(\rho_{h,0}-\bar{\rho}_{h,0})(u_{h,0}-\bar{u}_{h,0})dy,\label{E11}\\
&(g_{k-1,h}(x)-\chi^{(k-1)}(x))(\bar{\rho}_h \bar{u}_h^2+\bar{p}_h)+\sum_{i=1}^{k-1}E_{2,i}(h,\theta)\nonumber\\
=&-{y}^{(0)}(\bar{\rho}_{h,0}\bar{u}_{h,0}^2+\bar{p}_{h,0})
+\sum_{i=1}^{k-1}(b_{i-1}-s^{(i-1)})\int\limits_{(i-1)h}^{ih}\bar{p}_hd\tau+(b_{k-1}
-s^{(k-1)})\int\limits_{(k-1)h}^{x}\bar{p}_hd\tau\nonumber\\
&-\int\limits_{\chi^{(k-1)}(x)}^{g_{k-1,h}(x)}(\rho_{h,\theta} u_{h,\theta}-\overline{\rho_{h} u_h})(u_{h,\theta}-\bar{u}_h)dy
-\bar{u}_h\int\limits_{\chi^{(k-1)}(x)}^{g_{k-1,h}(x)}(\rho_{h,\theta}-\bar{\rho}_h)(u_{h,\theta}-\bar{u}_h)dy\nonumber\\
&+\int\limits_{{y}^{(0)}}^{0}(\rho_{h,0} u_{h,0}-\overline{\rho_{h,0} u_{h,0}})(u_{h,0}-\bar{u}_{h,0})dy
+\bar{u}_{h,0}\int\limits_{{y}^{(0)}}^{0}(\rho_{h,0}-\bar{\rho}_{h,0})(u_{h,0}-\bar{u}_{h,0})dy\nonumber \\
&+\sum_{i=1}^{k-1}\int\limits_{(i-1)h}^{ih}\Big(
(b_{i-1}-s^{(i-1)})(p_{h,\theta}|_{y=\chi^{(i-1)}}-\bar{p}_h)+b_{i-1}(p_{h,\theta}|_{y=g_{i-1}}
-p_{h,\theta}|_{y=\chi^{(i-1)}})\Big)d\tau \nonumber\\
&+\int\limits_{(k-1)h}^{x}\Big(
(b_{k-1}-s^{(k-1)})(p_{h,\theta}|_{y=\chi^{(k-1)}}-\bar{p}_h)+b_{k-1}(p_{h,\theta}|_{y=g_{k-1}}-
p_{h,\theta}|_{y=\chi^{(k-1)}})\Big)d\tau,\label{E21}\\
&(g_{k-1,h}(x)-\chi^{(k-1)}(x))\bar{\rho}_h \bar{u}_h(\frac{\gamma \bar{p}_h}{(\gamma-1)\bar{\rho}_h}+\frac{1}{2}\bar{u}_h^2)+\sum_{i=1}^{k-1}E_{3,i}(h,\theta)\nonumber \\
=&-{y}^{(0)}\bar{\rho}_{h,0} \bar{u}_{h,0}(\frac{\gamma \bar{p}_{h,0}}{(\gamma-1)\bar{\rho}_{h,0}}+\frac{1}{2}\bar{u}_{h,0}^2)
-\frac{\bar{u}_h^2}{2}\int\limits_{\chi^{(k-1)}(x)}^{g_{k-1,h}(x)}(\rho_{h,\theta}
-\bar{\rho}_h)(u_{h,\theta}-\bar{u}_h)dy\nonumber \\
&-\frac{1}{2}\overline{(u_{h}-\bar{u}_h)^2+v_{h}^2}\int\limits_{\chi^{(k-1)}(x)}^{g_{k-1,h}(x)}\rho_{h,\theta} u_{h,\theta} dy
-\frac{\gamma}{\gamma-1}\int\limits_{\chi^{(k-1)}(x)}^{g_{k-1,h}(x)}(p_{h,\theta}
-\bar{p}_h)(u_{h,\theta}-\bar{u}_h)dy\nonumber \\
&-\frac{1}{2}
\int\limits_{\chi^{(k-1)}(x)}^{g_{k-1,h}(x)}(\rho_{h,\theta} u_{h,\theta}-\overline{\rho_{h} u_{h}})(u_{h,\theta}^2+v_{h,\theta}^2-\overline{u_h^2+v_h^2})
dy+\frac{\bar{u}_{h,0}^2}{2}\int\limits_{{y}^{(0)}}^{0}(\rho_{h,0}-\bar{\rho}_{h,0})(u_{h,0}-\bar{u}_{h,0})dy\nonumber \\
&+\frac{1}{2}\overline{(u_{h,0}-\bar{u}_{h,0})^2+v_{h,0}^2}\int\limits_{{y}^{(0)}}^{0}\rho_{h,0} u_{h,0} dy
+\frac{\gamma}{\gamma-1}\int\limits_{{y}^{(0)}}^{0}(p_{h,0}-\bar{p}_{h,0})(u_{h,0}-\bar{u}_{h,0})dy\nonumber \\
&+\frac{1}{2}\int\limits_{{y}^{(0)}}^{0}(\rho_{h,0} u_{h,0}-\overline{\rho_{h,0} u_{h,0}})
(u_{h,0}^2+v_{h,0}^2-\overline{u_{h,0}^2+v_{h,0}^2})dy,\label{E31}\\
\mbox{and}\nonumber\\
&(g_{k-1,h}(x)-\chi^{(k-1)}(x))(\bar{\rho}_h \bar{u}_h \bar{Z}_h)+\sum_{i=1}^{k-1}E_{4,i}(h,\theta)\nonumber\\
=&-{y}^{(0)}\bar{\rho}_{h,0}\bar{u}_{h,0}\bar{Z}_{h,0}
-\int\limits_{\chi^{(k-1)}(x)}^{g_{k-1,h}(x)}(\rho_{h,\theta} u_{h,\theta}-\overline{\rho_h u_h})(Z_{h,\theta}-\bar{Z}_h)dy
-\bar{Z}_h\int\limits_{\chi^{(k-1)}(x)}^{g_{k-1,h}(x)}(\rho_{h,\theta}-\bar{\rho}_h)(u_{h,\theta}-\bar{u}_h)dy\nonumber\\
&+\int\limits_{{y}^{(0)}}^{0}(\rho_{h,0} u_{h,0}-\overline{\rho_{h,0} u_{h,0}})(Z_{h,0}-\bar{Z}_{h,0})dy
+\bar{Z}_{h,0}\int\limits_{{y}^{(0)}}^{0}(\rho_{h,0}-\bar{\rho}_{h,0})(u_{h,0}-\bar{u}_{h,0})dy.\label{E41}
\end{align}

Therefore, in order to derive the equations that the integral average of the solution of \eqref{eq-reaction-1} satisfies, 
we need to analyze the error terms at the right hand side of \eqref{E11}-\eqref{E41} as $h\to 0$ first.

By Theorem \ref{theorem-exist}, it holds that the terms like
$\int_{\chi(\tau)}^{g(\tau)}(\rho-\bar{\rho})(u-\bar{u})dy$
can be bounded by the square of the total variation of the weak solution, \emph{i.e.},
$$
\int_{\chi(\tau)}^{g(\tau)}(\rho-\bar{\rho})(u-\bar{u})dy= O(1)\delta_*^2,
$$
with $\delta_*$ in Theorem \ref{theorem-Quasi-One}. Next, from the decay property of the reactant $Z$, \emph{i.e.}, Lemma \ref{lemma-Re} and \eqref{3.11x}, we know that
\begin{align*}
&\int\limits_{0}^{x}\int\limits_{\chi(\tau)}^{g(\tau)}\rho \phi(T)Z dyd\tau-\int\limits_0^x (g(\tau)-\chi(\tau))\bar{\rho} \phi(\bar{T})\bar{Z}d\tau\\
=&\int\limits_{0}^{x}\int\limits_{\chi(\tau)}^{g(\tau)} \rho (\phi(T)-\phi(\bar{T}))Z dyd\tau+\int\limits_{0}^{x}\int\limits_{\chi(\tau)}^{g(\tau)}
(\rho-\bar{\rho})\phi(\bar{T})(Z-\bar{Z})dyd\tau\\
=&O(1)\delta_*^2.
\end{align*}

Therefore, we only need to estimate the last two terms in the right hand side of \eqref{E21}. To do that, we will carefully derive several estimates on the approximate strong contact discontinuity. Using the notations in the proof of Proposition \ref{prop-f}, we define $Q_{h,\theta}(\Lambda)$ based on the location of $\Lambda$,
\begin{equation}
Q_{h,\theta}(\Lambda)=
\begin{cases}
Q(\Lambda)                             & \text{for Case 1},\\
|\omega_k|+\sum_{i=2}^{5}|\alpha_i|    & \text{for Case 2},\\
\sum_{i=1}^{4}|\beta_i|                & \text{for Case 3.1},\\
\sum_{i=2}^{5}|\alpha_i|               & \text{for Case 3.2}.
\end{cases}
\end{equation}

Let $\Lambda_b=\cup_{k=1}^{+\infty}\Lambda_{k,0}$, where $\Lambda_{k,0}$ is the diamond centered at $P_k$. Let $L_{h,\theta}^b(\Lambda_b)$ be the summation of the strengths of the $1$-waves leaving $\Lambda_{b}$.

Similarly, let $\Lambda_c=\cup_{k=1}^{+\infty}\Lambda_{k,n_k}$, where $\Lambda_{k,n_k}$ is the diamond covering the strong contact discontinuity. Let
$L_{h,\theta}^c(\Lambda_c)$ be the summation of the strengths of the $5$-waves leaving $\Lambda_{c}$. Then, by \eqref{estimate-in-c1}, \eqref{estimate-in-c2}, and  \eqref{estimate-Fc}, we have
\begin{lemma}
There exists a constant M, independent of $U_{h,\theta}$, $\theta$, and $h$, such that
\begin{equation}\label{estimate-Q}
\sum_{\Lambda}Q_{h,\theta}(\Lambda)\leq M,\quad L_{h,\theta}^b(\Lambda_b)\leq M,\quad L_{h,\theta}^c(\Lambda_c)\leq M.
\end{equation}
where the summation is over all the diamonds $\Lambda$.
\end{lemma}

Next, let $\theta\in \prod_{k=1}^{\infty}(-1,1)\backslash \mathscr{N}$ be equidistributed, then we will prove the following lemma.
\begin{lemma}\label{IntTV}
There exists a positive constant $C$, such that
\[
\int_{0}^{+\infty}T.V.\{(\frac{v(\tau,\cdot)}{u(\tau,\cdot)},\, p(\tau,\cdot))\big|_{(\chi(\tau),g(\tau))}\}d\tau\leq C\delta_*.
\]
\end{lemma}
\begin{proof}
Since the velocity ratio $v_{h,\theta}/u_{h,\theta}$ and the pressure $p_{h,\theta}$ are invariant across the contact discontinuity, we only need to estimate the strengths of the weak $1$-wave and the weak $5$-wave. As in \cite{GlimmLax1970}, we denote by $dQ_{h,\theta}$ the measure assigning to $Q_{h,\theta}(\Lambda)$, 
and by $dL_{h,\theta}^{b}$ and $dL_{h,\theta}^{c}$ the measure assigning to $L_{h,\theta}^b(\Lambda_b)$ and $L_{h,\theta}^c(\Lambda_c)$, respectively.

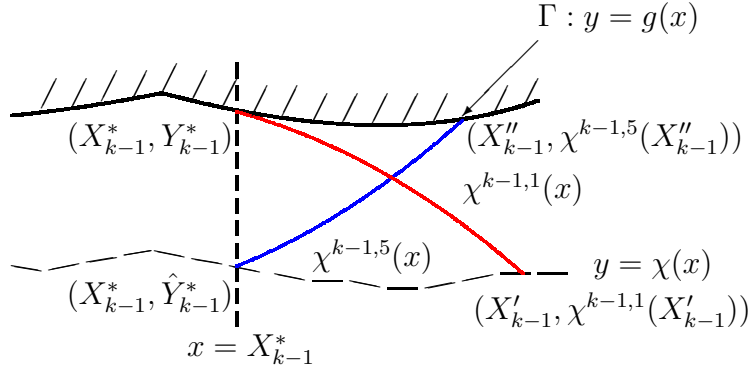
\begin{figure}[h]
\begin{center}
\setlength{\unitlength}{1mm}
\begin{picture}(120,40)(15,3)
\linethickness{1pt}
\qbezier(30,30)(40,31)(50,33)
\qbezier(50,33)(80,25)(100,32)
\thinlines
\put(30,10){\line(6,-1){4}}
\put(35,9.5){\line(5,1){4}}
\put(40,10.5){\line(6,1){4}}
\put(45,11.5){\line(6,1){3.6}}
\put(49,12){\line(5,-1){4}}
\put(54,11){\line(6,-1){4}}
\put(59,10){\line(5,-1){4}}
\put(65,9){\line(5,-1){4}}
\put(70,8){\line(1,0){4}}
\put(75,8){\line(5,-1){4}}
\put(80,7){\line(1,0){4}}
\put(85,7){\line(5,1){4}}
\put(90,8){\line(5,1){4}}
\put(95,9){\line(1,0){4}}
\put(100,9){\line(1,0){4}}
\thicklines
\multiput(60,2)(0,3){12}{\line(0,1){2}}
\color{blue}
\qbezier(60,10)(75,16)(90,29.4)
\color{red}
\qbezier(60,30.5)(80,25)(98,9)
\thinlines
\color{black}
\put(34,31){\line(1,2){2}}
\put(38,31){\line(1,2){2}}
\put(42,32){\line(1,2){2}}
\put(46,32){\line(1,2){2}}
\put(50,33){\line(1,2){2}}
\put(54,32){\line(1,2){2}}
\put(58,31){\line(1,2){2}}
\put(62,30){\line(1,2){2}}
\put(66,30){\line(1,2){2}}
\put(70,29){\line(1,2){2}}
\put(74,29){\line(1,2){2}}
\put(78,29){\line(1,2){2}}
\put(82,29){\line(1,2){2}}
\put(86,29){\line(1,2){2}}
\put(94,30){\line(1,2){2}}
\put(98,31){\line(1,2){2}}
\put(36,26){ $(X_{k-1}^*,Y_{k-1}^*)$}
\put(36,5){ $(X_{k-1}^*,\hat{Y}_{k-1}^*)$ }
\put(90,3){ $(X'_{k-1},\chi^{k-1,1}(X'_{k-1}))$}
\put(89,26){ $(X''_{k-1},{\chi}^{k-1,5}(X''_{k-1}))$}

\put(100,40){\vector(-1,-1){10}}
\put(100,42){$\Gamma:y=g(x)$}
\put(52,-2){ $x=X_{k-1}^*$}
\put(106,9){ $y=\chi(x)$}
\put(90,19){$\chi^{k-1,1}(x)$}
\put(70,10){$\chi^{k-1,5}(x)$}
\end{picture}
\end{center}
\caption{Generalized characteristics in $U_{h,\theta}$.}\label{15}
\end{figure}

As shown in Fig. \ref{15}, let the line $x=X_{k-1}^*$ intersect $\partial\Omega$ and $y=\chi(x)$ at $(X_{k-1}^*,Y_{k-1}^*)$ and $(X_{k-1}^*,\hat{Y}_{k-1}^*)$ respectively. 
Let $y=\chi^{k-1,1}(x)$ be the maximal $1$-generalized characteristics in $U_{h,\theta}$ emanating from the point $(X_{k-1}^*,Y_{k-1}^*)$, and let $y=\chi^{k-1,5}(x)$ be the minimum $5$-generalized characteristics in $U_h$ emanating from the point $(X_{k-1}^*,\hat{Y}_{k-1}^*)$.
Moreover, let $y=\chi^{k-1,1}(x)$ and $y={\chi}^{k-1,5}(x)$ intersect $y=\chi(x)$ and $\partial\Omega$ at $(X'_{k-1},\chi^{k-1,1}(X'_{k-1}))$ and $(X''_{k-1},{\chi}^{k-1,5}(X''_{k-1}))$ respectively for some $X'_{k-1}$ and $X''_{k-1}$. Thus, by Lemma \ref{lemma-g-chi}, there exists a constant $X^*>0$, independent of $X_{k-1}^*$, such that  $X_{k-1}^{*}+X^*$ is greater than $X'_{k-1}$ and $X''_{k-1}$. Then we get a sequence $\{X_{k}^{*}\}_{k=0}^{\infty}$ by setting $X_{k}^{*}=X_{k-1}^{*}+X^*$.

we denote by $\Omega_{k-1}^{*}$ the domain with the boundaries that $x=X_{k-1}^*$, $x=X_{k}^*$, $\partial\Omega_h$, and $y=\chi_{h,\theta}(x)$. Let $\tilde{L}_{1,h}(X-)$ (or $L_{1,h}(X-)$) be the summation of all the strength of the weak $1$-waves after (or before) the reaction step on the line $x=X$. Obviously for $ih<X\leq(i+1)h$,
$$
\tilde{L}_{1,h}(X-)-L_{1,h}(X-)\leq M e^{-lih}h\|Z_{0}\|_{\infty}.
$$
Then by the equations that the approximate solutions satisfy, we can deduce in the same way as the one in \cite{GlimmLax1970,Liu1977Behavior} that 
on the line $x=X_{k}^*$, if $h_j$ is sufficiently small, then
\[
\tilde{L}_{1,h_j}(X_{k}^*-)=O(1)(dL^{b}_{h_j,\theta}(\Lambda_{b,k-1}^{*})+dQ_{h_j,\theta}(\Lambda_{k-1}^{*})
+(e^{-lX_{k-1}^*}-e^{-lX_{k}^*})\|Z_{0}\|_{\infty}),
\]
where $\Lambda_{b,k-1}^{*}$ consists of the diamonds covering $\partial\Omega_{k-1}^{*}\cap\partial\Omega_h$, $\Lambda_{k-1}^{*}$ consists of the diamonds in the interior of $\Omega_{k-1}^{*}$, and the bound of $O(1)$ is independent of $U_{h_j}$ and $h_j$.

Similarly, let $\tilde{L}_{5,h_j}(X-)$ stand for the summation of all the strength of the weak $5$-waves on the line $x=X$ after the reaction step, then on the line $x=X_{k}^*$, if $h_j$ is sufficiently small, then
\[
\tilde{L}_{5,h_j}(X_{k}^*-)=O(1)(dL^{c}_{h_j,\theta}(\Lambda_{c,k-1}^{*})+dQ_{h_j,\theta}(\Lambda_{k-1}^{*})
+(e^{-lX_{k-1}^*}-e^{-lX_{k}^*})\|Z_{0}\|_{\infty}),
\]
where $\Lambda_{c,k-1}^{*}$ consists of the diamonds covering the strong contact discontinuity $y=\chi_{h,\theta}(x)$ in $\Omega_{k-1}^{*}$.

Next, for $x\in(X_{k-1}^*,X_{k}^*)$, it follows from the local estimates in Section \ref{sec:LE} that
\begin{align*}
&\tilde{L}_{1,h_j}(x-)=O(1)(\tilde{L}_{1,h_j}(X_{k-1}^*+)+dL^{b}_{h_j,\theta}(\Lambda_{b,k-1}^{*})
+dQ_{h_j,\theta}(\Lambda_{k-1}^{*})
+(e^{-lX_{k-1}^*}-e^{-lX_{k}^*})\|Z_{0}\|_{\infty}),\\
&\mbox{and}\\
&\tilde{L}_{5,h_j}(x-)=O(1)(\tilde{L}_{5,h_j}(X_{k-1}^*+)+dL^{c}_{h_j,\theta}(\Lambda_{c,k-1}^{*})
+dQ_{h_j,\theta}(\Lambda_{k-1}^{*})
+(e^{-lX_{k-1}^*}-e^{-lX_{k}^*})\|Z_{0}\|_{\infty}).
\end{align*}

Therefore, in domain $\Omega_{k-1}^{*}$, we have
\begin{align*}
&\int\limits_{X_{k-1}^{*}}^{X_{k}^{*}}T.V.\{(\frac{v_{h_j}(\tau,\cdot)}{u_{h_j}(\tau,\cdot)},p_{h_j}(\tau,\cdot))\big|_{(\chi_{i-1},g_{i-1})}\}d\tau\\
\leq& O(1)(X^*_k-X^*_{k-1})\max_{x\in[X^*_{k-1},X^*_{k}]}(\tilde{L}_{1,h_j}(x-)+\tilde{L}_{5,h_j}(x-))\\
\leq& O(1)X^*\big(\tilde{L}_{1,h_j}(X_{k-1}^*-)+\tilde{L}_{5,h_j}(X_{k-1}^*-)+dL^{b}_{h_j,\theta}
(\Lambda_{b,k-1}^{*})+dL^{c}_{h_j,\theta}(\Lambda_{c,k-1}^{*})\\
&\hspace{4em}+dQ_{h_j,\theta}(\Lambda_{k-1}^{*})+(e^{-lX_{k-1}^*}-e^{-lX_{k}^*})\|Z_{0}\|_{\infty}\big).
\end{align*}

Then by \eqref{estimate-Q}, we complete the proof.

\end{proof}

Now, by applying Theorem \ref{theorem-exist} and Lemma \ref{lemma-tildeE} and passing the limit $h_j\rightarrow0$, we obtain that the equations satisfied by the integral average of the weak solution of \eqref{eq-reaction-1} are
\begin{align}
(g&(x)-\chi(x))\bar{\rho} \bar{u}=-{y}^{(0)}\bar{\rho}_{0}\bar{u}_{0}+O(1)\delta_*^2,\label{averge-w-1}\\
(g&(x)-\chi(x))(\bar{\rho} \bar{u}^2+\bar{p})=-{y}^{(0)}(\bar{\rho}_{0}\bar{u}_{0}^2+\bar{p}_{0})+\int_{0}^{x}(g'(\tau)-\chi'(\tau))\bar{p}d\tau
+O(1)\delta_*^2,\label{averge-w-2}\\
(g&(x)-\chi(x))\bar{\rho} \bar{u}(\frac{\gamma \bar{p}}{(\gamma-1)\bar{\rho}}+\frac{1}{2}\bar{u}^2)\nonumber\\
=-&{y}^{(0)}\bar{\rho}_{0} \bar{u}_{0}(\frac{\gamma \bar{p}_{0}}{(\gamma-1)\bar{\rho}_{0}}+\frac{1}{2}\bar{u}_{0}^2)
+q_0\int_0^x (g(\tau)-\chi(\tau))\bar{\rho}\phi(\bar{T})\bar{Z}d\tau
+O(1)\delta_*^2,\label{averge-w-3}\\
\mbox{and}\nonumber\\
(g&(x)-\chi(x))(\bar{\rho} \bar{u} \bar{Z})=-{y}^{(0)}\bar{\rho}_{0}\bar{u}_{0}\bar{Z}_{0}-\int_0^x (g(\tau)-\chi(\tau))\bar{\rho} \phi(\bar{T})\bar{Z}d\tau+O(1)\delta_*^2.\label{averge-w-4}
\end{align}

\subsection{Proof of Theorem \ref{theorem-Quasi-One}}
Finally, we can show Theorem \ref{theorem-Quasi-One} now.
\begin{proof}
Let $A(x)=g(x)-\chi(x)$, and let $A(0)=-{y}^{(0)}$, then equations \eqref{averge-w-1}-\eqref{averge-w-4} become
\begin{equation}\label{4.37s}
\begin{cases}
&\bar{\rho} \bar{u}A(x)=\bar{\rho}_{0}\bar{u}_{0}A(0)+O(1)\delta_*^2,\\
&(\bar{\rho} \bar{u}^2+\bar{p})A(x)=(\bar{\rho}_{0}\bar{u}_{0}^2+\bar{p}_{0})A(0)+\int\limits_{0}^{x}A'(\tau)\bar{p}d\tau
+O(1)\delta_*^2,\\
&(\frac{\gamma \bar{p}}{(\gamma-1)\bar{\rho}}+\frac{1}{2}\bar{u}^2)\bar{\rho} \bar{u}A(x)
=(\frac{\gamma \bar{p}_{0}}{(\gamma-1)\bar{\rho}_{0}}+\frac{1}{2}\bar{u}_{0}^2)\bar{\rho}_{0} \bar{u}_{0}A(0)
+q_0\int_0^x A(\tau)\bar{\rho}\phi(\bar{T}) \bar{Z}d\tau+O(1)\delta_*^2,\\
&\bar{\rho} \bar{u} \bar{Z}A(x)=\bar{\rho}_{0}\bar{u}_{0}\bar{Z}_{0}A(0)-\int_0^x A(\tau)\bar{\rho}\phi(\bar{T})\bar{Z}d\tau+O(1)\delta_*^2.
\end{cases}
\end{equation}
On the other hand, by Lemma \ref{lemma-Quasi-One-Ex} and Lemma \ref{IntTV}, system \eqref{eq-QO2}
admits a unique solution $U_A(x)=(\rho_{A},u_{A},p_{A},Z_{A})^\top$ satisfying \eqref{estimate-QO}.

By the straightforward calculation from the fourth equation of \eqref{4.37s}, we have that
\begin{align*}
&\bar{Z} =\bar{Z}_{0}\exp(-\frac{1}{\bar{\rho}_{0} \bar{u}_{0} A(0)}\int_0^x A(\tau)\bar{\rho}\phi(\bar{T})d\tau)+O(1)\delta_*^2.
\end{align*}
Similarly, from the fourth equation of \eqref{eq-QO2}, we have
\begin{align*}
&Z_A =\bar{Z}_{0}\exp(-\frac{1}{\bar{\rho}_{0} \bar{u}_{0} A(0)}\int_0^x A(\tau)\rho_A\phi(T_A)d\tau).
\end{align*}
Then
\[
|\bar{Z}-Z_A|\leq O(1)\delta_0\max_{0\leq\tau\leq x}(|\bar{\rho}-\rho_A|+|\bar{T}-T_A|)+O(1)\delta_*^2.
\]

Next, from the first three equations of the two systems \eqref{eq-QO2} and \eqref{4.37s}, we have
\begin{equation*}
\begin{cases}
&(\bar{\rho}-\rho_A) \bar{u} A(x)+ \rho_A (\bar{u}-u_A) A(x)=O(1)\delta_*^2,\\
&(\bar{\rho}-\rho_A) \bar{u}^2 A(x)+\rho_A(\bar{u}^2-u_A^2)A(x)+(\bar{p}-p_A)A(x)
=\int_0^x A'(\tau)(\bar{p}-p_A)d\tau+O(1)\delta_*^2,\\
&\frac{\gamma}{\gamma-1}(\bar{p}-p_A)\bar{u}A(x) +\frac{\gamma}{\gamma-1}p_A(\bar{u}-u_A) A(x)+\frac{1}{2}(\bar{\rho}-\rho_A)\bar{u}^3A(x) +\frac{1}{2}\rho_A(\bar{u}^3-u_A^3)A(x)\\
&=q_0\int_0^x A(\tau)(\bar{\rho}\phi(\bar{T})\bar{Z}-\rho_A\phi(T_A)Z_A)d\tau+O(1)\delta_*^2.
\end{cases}
\end{equation*}

From Lemma \ref{IntTV}, we easily have the following fact that
\begin{align*}
&\Big|\int_0^x A'(\tau)(\bar{p}-p_A)d\tau\Big|\leq O(1)\delta_*\max_{0\leq\tau\leq x}|\bar{p}-p_A|,
\end{align*}
and from Lemma \ref{lemma-g-chi} and the estimates on the error terms in Section \ref{subsec:4.3s}, we also have that
\begin{align*}
&\Big|\int_0^x A(\tau)\big(\bar{\rho}\phi(\bar{T})\bar{Z}-\rho_A\phi(T_A)Z_A\big)d\tau\Big|\\
&\leq A(x)(\rho_A u_A Z_A- \bar{\rho} \bar{u} \bar{Z})+O(1)\delta_*^2,\\
&\leq O(1)\delta_*\max_{0\leq\tau\leq x}(|\bar{\rho}-\rho_A|+|\bar{T}-T_A|+|\bar{u}-u_A|)+O(1)\delta_*^2.
\end{align*}
Therefore, it follows from the implicit function theorem that there exists a constant $C>0$, such that
\[
\max_{x\geq 0}|\bar{U}-U_A|\leq C\delta_*^2,
\]
for sufficiently small $\delta_*$. This completes the proof.
\end{proof}


\setcounter{figure}{0}
\renewcommand{\thefigure}{A.\arabic{figure}}


\textbf{Ackowledgments}: The research of  Wei Xiang was supported in part by the CityU Start-Up Grant for New Faculty 7200429(MA), and the Research Grants Council of the HKSAR, China (Project No. CityU 21305215 and Project No. CityU 11332916). The research of Yongqian Zhang and Qin Zhao was supported in part by NSFC Project 11421061, NSFC Project 11031001, NSFC Project 11121101, the 111 Project B08018 (China),and the Shanghai Natural Science Foundation 15ZR1403900.

\bibliographystyle{plain}

\end{document}